\documentclass{amsart}
\usepackage{graphicx}
\usepackage{amssymb}
\usepackage{amsmath}
\usepackage{amsthm}
\usepackage{subfigure}
\usepackage{enumerate}
\usepackage{bbm}
\usepackage{color}

\newtheorem{thm}{Theorem}[section]
\newtheorem{cor}[thm]{Corollary}

\newtheorem{lem}[thm]{Lemma}
\newtheorem{prop}[thm]{Proposition}
\theoremstyle{definition}
\newtheorem{defn}[thm]{Definition}
\newtheorem{rem}[thm]{Remark}

\newtheorem*{defn*}{Definition}
\newtheorem*{rems*}{Remarks}
\newtheorem*{rem*}{Remark}

\newtheorem{con}[thm]{Conjecture}

\numberwithin{equation}{section}

\newcommand{\overarc}{\overline}

\newcommand{\Eq}{{E}}
\newcommand{\E}{\mathbb{E}}

\newcommand{\Css}{{CSS}}

\newcommand{\C}{\mathcal{C}}

\newcommand{\no}[1]{\#\left\{#1\right\}}

\begin{document}

\title[The Gauss-Bonnet theorem] {The Gauss-Bonnet Theorem for coherent tangent bundles over surfaces with boundary \linebreak and its applications}
\author{Wojciech Domitrz, Micha\l{} Zwierzy\'nski}
\address{Faculty of Mathematics and Information Science\\
Warsaw University of Technology\\
ul. Koszykowa 75 \\
00-662 Warszawa}

\email{domitrz@mini.pw.edu.pl, zwierzynskim@mini.pw.edu.pl}
\thanks{The work of W. Domitrz and M. Zwierzy\'nski was partially supported by NCN grant no. DEC-2013/11/B/ST1/03080. }



\begin{abstract}
In \cite{SUY2, SUY, SUY3} the Gauss-Bonnet formulas for coherent tangent bundles over compact oriented surfaces (without boundary) were proved. We establish the Gauss-Bonnet theorem for coherent tangent bundles over compact oriented surfaces with boundary. We apply this theorem to investigate global properties of maps between  surfaces with boundary. As a corollary of our results we obtain Fukuda-Ishikawa's theorem. We also study geometry of the affine extended wave fronts for planar closed non singular hedgehogs (rosettes). In particular, we find a link between the total geodesic curvature on the boundary and the total singular curvature of the affine extended wave front, which leads to a relation of integrals of functions of the width of a resette.
\end{abstract}

\maketitle

\section{Introduction}\label{SectionIntro}

The local and global geometry of fronts and coherent tangent bundles, which are natural generalizations of fronts, has been recently very carefully studied in \cite{KRSUY, MS1, MSUY1, SUY2, SUY, SUY3, SUY4}. In particular in \cite{SUY2, SUY} the results of M. Kossowski (\cite{Kos1, Kos2}) and  R.~Langevin, G.~Levitt, H.~Rosenberg (\cite{LLR2}) were generalized to the following  Gauss-Bonnet type formulas for the singular coherent tangent bundle $\mathcal E$ over a compact surface $M$ whose set of singular points $\Sigma$ admits at most peaks:
\begin{align}
\label{GBintro1} 2\pi\chi(M)&=\int_{M}KdA+2\int_{\Sigma}\kappa_sd\tau,\\
\label{GBintro2} \frac{1}{2\pi}\int_MKd\hat{A}& = \chi(M^+)-\chi(M^-)+\# P^+-\# P^-.
\end{align}
In the above formulas $K$ is the Gaussian curvature, $\kappa_s$ is the singular curvature, $d\tau$ is the arc length measure on $\Sigma$, $d\hat{A}$ (respectively $dA$) is the signed (respectively unsigned) area form,
$M^+$ (respectively $M^-$) is the set of regular points in $M$, where $d\hat{A}=dA$ (respectively $d\hat{A}=-dA$), $P^+$ (respectively $P^-$) is the set of positive (respectively negative) peaks (see \cite{SUY2} and Section \ref{SectionGaussBonnet} for details).  K. Saji, M. Umeraha and K. Yamada also found several interesting applications of the above formulas (see especially \cite{SUY3}).

The classical Gauss-Bonnet theorem was formulated for compact oriented surfaces with boundary. Therefore it is natural to find the analogous Gauss-Bonnet formulas for coherent tangent bundles over compact oriented surfaces with boundary (see Theorem \ref{ThmGaussBonnetFrontsWithB}). Coherent tangent bundles over compact oriented surfaces with boundary also appear in many problems. In this paper we apply the Gauss-Bonnet formulas to study smooth maps between compact oriented surfaces with boundary and affine extended  wave fronts of the planar non-singular hedgehogs (rosettes).  As a result, we obtain a new proof of Fukuda-Ishikawa's theorem (\cite{FI}) and we find a link between the total geodesic curvature on the boundary and the total singular curvature of the affine extended wave front of a rosette.. This leads to a relation between the integrals of the function of the width of the rosette, in particular of the width of an oval (see Theorem \ref{ThmIntegralWidth} and Conjecture \ref{ConIntegralWidth}).

In Section \ref{SectionGaussBonnet} we briefly sketch the theory of coherent tangent bundles and state the Gauss-Bonnet theorem for coherent tangent bundles over compact oriented surfaces with boundary (Theorem \ref{ThmGaussBonnetFrontsWithB}), which is the main result of this paper. The proof of Theorem \ref{ThmGaussBonnetFrontsWithB} is presented in Section \ref{SectionProof}. We apply this theorem to study the global properties of maps between compact oriented surfaces with boundary in Section \ref{GenApplications}. The last section contains the results on the geometry of the affine extended wave fronts of rosettes.


\section{The Gauss-Bonnet theorem}\label{SectionGaussBonnet}

In this section we formulate the Gauss-Bonnet type theorem for coherent tangent bundles over compact oriented surfaces with boundary. The proof of this theorem is presented in the next section. Coherent tangent bundles are intrinsic formulation of wave fronts.  The theory of coherent tangent bundles were introduced and developed in \cite{SUY2, SUY, SUY3}. We recall basic definitions and facts of this theory  (for details see \cite{SUY2, SUY3}).

\begin{defn}Let $M$ be a $2$-dimensional compact oriented surface (possibly with boundary).  A \textit{coherent tangent bundle} over $M$ is a $5$-tuple $\left(M, \mathcal{E}, \left<\cdot,\cdot\right>, D, \psi\right)$, where $\mathcal{E}$ is an orientable vector bundle over $M$ of rank $2$, $\left<\cdot,\cdot\right>$ is a metric, $D$ is a metric connection on $(\mathcal{E}, \left<\cdot,\cdot\right>)$ and $\psi$ is a bundle homomorphism
\begin{align*}
\psi: TM\to\mathcal{E},
\end{align*}
such that for any smooth vector fields $X$, $Y$ on $M$
\begin{align}\label{coherent}
D_X\psi(Y)-D_Y\psi(X)=\psi([X,Y]).
\end{align}
\end{defn}

The pull-back metric $ds^2:=\psi^{\ast}\left<\cdot,\cdot\right>$ is called the {\it first fundamental form } on $M$. Let $\mathcal{E}_p$ denote the fiber of $\mathcal{E}$ at a point $p\in M$. If $\psi_p:=\psi|_{T_pM}:T_pM\to\mathcal{E}_p$ is not a bijection at a point $p\in M$, then $p$ is called a \textit{singular point}. Let $\Sigma$ denote the set of singular points on $M$. If a point $p\in M$ is not a singular point, then $p$ is called a \textit{regular point}. Let us notice that the first fundamental form on $M$ is positive definite at regular points and it is not positive definite at singular points.

Let $\mu\in\text{Sec}(\mathcal{E}^*\wedge\mathcal{E}^*)$ be a smooth non-vanishing skew-symmetric bilinear section such that for any orthonormal frame $\{e_1,e_2\}$ on $\mathcal{E}$ $\mu(e_1,e_2)=\pm 1$. The existence of such $\mu$ is a consequence of the assumption that $\mathcal{E}$ is orientable. A \textit{co-orientation} of the coherent tangent bundle is a choice of $\mu$. An orthonormal frame $\{e_1,e_2\}$ such that $\mu(e_1,e_2)=1$ (respectively $\mu(e_1,e_2)=-1$) is called \textit{positive} (respectively \textit{negative}) with respect to the co-orientation $\mu$.

From now on, we fix a co-orientation $\mu$ on the coherent tangent bundle.

\begin{defn}\label{SignedAreaDef}
Let $(U; u,v)$ be a positively oriented local coordinate system on $M$. Then $d\hat{A}:=\psi^*\mu=\lambda_{\psi}  du\wedge dv$ (respectively $dA:=|\lambda_{\psi} |du\wedge dv$) is called the \textit{signed area form} (respectively the \textit{unsigned area form}), where
\begin{align*}
\lambda_{\psi} :=\mu\left(\psi_u, \psi_v\right),  \psi_u:=\psi\left(\frac{\partial}{\partial u}\right), \psi_v:=\psi\left(\frac{\partial}{\partial v}\right).
\end{align*}
The function $\lambda_{\psi} $ is called the \textit{signed area density function} on $U$.
\end{defn}

The set of singular points on $U$ is expressed as
\begin{align*}
\Sigma\cap U:=\left\{p\in U: \lambda_{\psi} (p)=0\right\}.
\end{align*}

Let us notice that the signed and unsigned area forms, $d\hat{A}$ and $dA$, give globally defined $2$-forms on $M$ and they are independent of the choice of positively oriented local coordinate system $(u,v)$. Let us define

\begin{align*}
M^+: =\Big\{p\in M\setminus\Sigma\ \Big|\ d\hat{A}_p=dA_p\Big\},\ \ \
M^-: =\Big\{p\in M\setminus\Sigma\ \Big|\ d\hat{A}_p=-dA_p\Big\}.
\end{align*}

We say that a singular point $p\in\Sigma$ is \textit{non-degenerate} if $d\lambda_{\psi} $ does not vanish at $p$. Let $p$ be a non-degenerate singular point. There exists a neighborhood $U$ of $p$ such that the set $\Sigma\cap U$ is a regular curve, which is called the \textit{singular curve}. The \textit{singular direction} is the tangential direction of the singular curve. Since $p$ is non-degenerate, the rank of $\psi_p$ is $1$. The \textit{null direction} is the direction of the kernel of $\psi_p$. Let $\eta(t)$ be the smooth (non-vanishing) vector field along the singular curve $\sigma(t)$ which gives the null direction.

Let $\wedge$ be the exterior product on $TM$.

\begin{defn} Let $p\in M$ be a non-degenerate singular point and let $\sigma(t)$ be a singular curve such that $\sigma(0)=p$. The point $p$ is called an \textit{$A_2$-point} (or an \textit{intrinsic cuspidal edge}) if the null direction at $p$ (i.e. $\eta(0)$) is transversal to the singular direction at $p$ (i.e. $\dot{\sigma}(0):=\frac{d\sigma}{dt}\big|_{t=0}$). The point $p$ is called an \textit{$A_3$-point}  (or an \textit{intrinsic swallowtail}) if the point $p$ is not an \textit{$A_2$-point} and
\begin{align*}
\frac{d}{dt}(\dot{\sigma}(t)\wedge \eta(t))|_{t=0}\ne 0.
\end{align*}
\end{defn}

\begin{defn}  Let $p$ be a singular point $p\in M$ which is not an $A_2$-point. The point $p$ is called a \textit{peak} if there exists a coordinate neighborhood $(U; u,v)$ of $p$ such that:
\begin{enumerate}[(i)]
\item if  $q\in (\Sigma\cap U)\setminus\{p\}$ then $q$ is an $A_2$-point;
\item the rank of the linear map $\psi_p: T_pM\to\mathcal{E}_p$ at $p$ is equal to $1$;
\item the set $\Sigma\cap U$ consists of finitely many $C^1$-regular curves emanating from $p$.
\end{enumerate}
\end{defn}

A peak is a non-degenerate if it is a non-degenerate singular point.

From now one we assume that the set of singular points $\Sigma$ \textit{admits at most peaks}, i.e. $\Sigma$ consists of $A_2$-points and peaks.

Furthermore let us fix a Riemannian metric $g$ on $M$. Since the first fundamental form $ds^2$ degenerates on $\Sigma$, there exists a $(1,1)$-tensor field $I$ on $M$ such that
\begin{align*}
ds^2(X,Y)=g(IX,Y),
\end{align*}
for smooth vector fields $X,Y$ on $M$. We fix a singular point $p\in\Sigma$. Since $\Sigma$ admits at most peaks, the point $p$ is an $A_2$ - point or a peak. Let $\lambda_1(p)$, $\lambda_2(p)$ be the eigenvalues of $I_p:=I\big|_{T_pM}:T_pM\to T_pM$. Since the kernel of $\psi_p$ is one dimensional, the only one of $\lambda_1(p)$, $\lambda_2(p)$ vanishes. Thus there exists a neighborhood $V$ of $p$ such that for every point $q\in V$ the map $I_q$ has eigenvalues $\lambda_1(q)$, $\lambda_2(q)$, such that $0\leqslant\lambda_1(q)<\lambda_2(q)$. Furthermore there exists a coordinate neighborhood $(U; u,v)$ of $p$ such that $U$ is a subset of $V$ and the $u$ - curves (respectively $v$ - curves) give the $\lambda_1$ - eigendirections (respectively $\lambda_2$ - eigendirections). Such a local coordinate system $(U; u,v)$ is called a \textit{$g$-coordinate system} at $p$.

\begin{defn}\label{DefInitial}
Let $\gamma(t)$ $(0\le t<1)$ be a $C^1$-regular curve on $M$ such that $\gamma(0)=p$. The $\mathcal{E}$-\textit{initial vector} of $\gamma$ at $p$ is the following limit
\begin{align}
\label{InitialVectorFormula}
\Psi_{\gamma}:=\lim_{t\rightarrow 0+}\frac{\psi(\dot{\gamma}(t))}{|\psi(\dot{\gamma}(t))|}\in \mathcal{E}_p
\end{align}
if it exists.
\end{defn}

\begin{rem}
If $p$ is a regular point of $M$ then the $\mathcal{E}$ - initial vector of $\gamma$ at $p$ is the unit tangent vector of $\gamma$ at $p$ with respect to the first fundamental form $ds^2$.
\end{rem}

\begin{prop}[Proposition 2.6 in \cite{SUY2}]\label{PropPeakLimit}
Let $\gamma$ be a $C^1$ - regular curve emanating from an $A_2$ - point or a peak $p$ such that $\dot{\gamma}(0)$ is a not a null vector or $\gamma$ is a singular curve. Then the $\mathcal{E}$ - initial vector of $\gamma$ at $p$ exists.
\end{prop}

Since we study coherent tangent bundles over surfaces with boundary, we also consider a curve $\gamma$ on the boundary which is tangent to the null direction at a singular point $p$ on the boundary. We prove that in this case the $\mathcal E$-initial vector of $\gamma$ at $p$ exists if the singular direction is transversal to the boundary at $p$.

\begin{prop}\label{PropLimit}
Let $\left(\mathcal{E}, \left<\cdot, \cdot\right>, D, \psi\right)$ be a coherent tangent bundle over an compact oriented surface $M$ with boundary. Let $p$ be an $A_2$-point in the boundary $\partial M$. If the boundary $\partial M$ is transversal to $\Sigma$ at $p$ and $\gamma:(-\varepsilon,\varepsilon)\to\partial M$ is a $C^2$-regular curve such that $\gamma(0)=p$, $\gamma\big((-\varepsilon,\varepsilon)\big)\cap\Sigma=\{p\}$ and $\dot{\gamma}(0)\in T_p\partial M$ is a null direction, then the $\mathcal{E}$-initial vector $\Psi_{\gamma}$ of $\gamma$ at $p$ exists, $D_{\frac{d}{dt}}\left.\left(\psi\big(\dot{\gamma}(t)\big)\right)\right|_{t=0}\ne 0$, and
\begin{align}\label{EqLimit}
\Psi_{\gamma}=\frac{D_{\frac{d}{dt}}\left.\left(\psi\big(\dot{\gamma}(t)\big)\right)\right|_{t=0}}{\left|D_{\frac{d}{dt}}\left.\left(\psi\big(\dot{\gamma}(t)\big)\right)\right|_{t=0}\right|}\in\mathcal{E}_p.
\end{align}
\end{prop}
\begin{proof}

Let $\sigma:[0,\varepsilon)\to\Sigma$ be a singular curve such that $\sigma(0)=p$. Let $(U; u, v)$ be a $g$-coordinate system at $p$ i.e. the null direction at $\sigma(t)$ is spanned by $\frac{\partial}{\partial u}$.
 Since $\lambda_{\psi}(\sigma(t))=0$, we get that
\begin{align}\label{Eqdlambdasigmaprime}
\frac{d}{dt}\left.\left(\lambda_{\psi}\big(\sigma(t)\big)\right)\right|_{t=0}=d\lambda_{\psi}\big|_p\cdot\dot\sigma(0)=0.
\end{align}
Let us notice that
\begin{align}\label{Eqdlambdagammaprime}
\frac{d}{dt}\left.\left(\lambda_{\psi}\big(\gamma(t)\big)\right)\right|_{t=0}=d\lambda_{\psi}\big|_p\cdot\dot\gamma(0) \neq 0
\end{align}
since the vectors $\dot\sigma(0)$ and $\dot\gamma(0)$ span the space $T_pM$ and $d\lambda_{\psi}\big|_p\neq 0$.

On the other hand, since $\lambda_{\psi}\big(\gamma(t)\big)=\mu\left(\psi_u\big(\gamma(t)\big), \psi_v\big(\gamma(t)\big)\right)$ and $\psi_u\big(\gamma(0)\big)=0$, we get the following:
\begin{align}
\label{EqMunonzero} &\frac{d}{dt}\left.\left(\lambda_{\psi}\big(\gamma(t)\big)\right)\right|_{t=0} =\frac{d}{dt}\left.\mu\left(\psi_u\big(\gamma(t)\big), \psi_v\big(\gamma(t)\big)\right)\right|_{t=0}=\\
\nonumber &=\mu\left(D_{\frac{d}{dt}}\left(\psi_u\big(\gamma(t)\big)\Big|_{t=0}\right), \psi_v\big(\gamma(0)\big)\right)
\end{align}
By \eqref{Eqdlambdagammaprime} and \eqref{EqMunonzero} we get that $D_{\frac{d}{dt}}\left(\psi_u\big(\gamma(t)\big)\Big|_{t=0}\right)$, $\psi_v\big(\gamma(0)\big)$ are linearly independent.

The vector field $\dot{\gamma}$ can be written in the following form $\displaystyle\dot{\gamma}(t)=\dot{u}(t)\frac{\partial}{\partial u}+\dot{v}(t)\frac{\partial}{\partial v}$, where $u(t)=t(a+h(t))$, $v(t)=t^2g(t)$, $a\neq 0$ and $h$, $g$ are some functions such that $h(0)=0$. Similarly since $\psi_u\big(\gamma(0)\big)=0$ and $D_{\frac{d}{dt}}\left.\left(\psi_u\big(\gamma(t)\big)\right)\right|_{t=0}\neq 0$ we can write $\psi_u\big(\gamma(t)\big)=t\xi(t)$, where $\xi(t)\in\mathcal E_{\gamma(t)}$ and $\xi(0)\neq 0$.

Now we will prove the formula \eqref{EqLimit}.

\begin{align*}
\lim_{t\to 0^+}\frac{\psi\big(\dot{\gamma}(t)\big)}{\left|\psi\big(\dot{\gamma}(t)\big)\right|}
&=\lim_{t\to 0^+}\frac{\dot{u}(t)\psi_u\big(\gamma(t)\big)+\dot{v}(t)\psi_v\big(\gamma(t)\big)}{\left|\dot{u}(t)\psi_u\big(\gamma(t)\big)+\dot{v}(t)\psi_v\big(\gamma(t)\big)\right|}\\
&=\lim_{t\to 0^+}\frac{t\left(\big(a+h(t)+t\dot{h}(t)\big)\xi(t)+\big(2g(t)+t\dot{g}(t)\big)\psi_v\big(\gamma(t)\big)\right)}{t\left|\big(a+h(t)+t\dot{h}(t)\big)\xi(t)+\big(2g(t)+t\dot{g}(t)\big)\psi_v\big(\gamma(t)\big)\right|}\\
&=\frac{a\xi(0)+2g(0)\psi_v\big(\gamma(0)\big)}{|a\xi(0)+2g(0)\psi_v\big(\gamma(0)\big)|},
\end{align*}

where the expression $a\xi(0)+2g(0)\psi_v\big(\gamma(0)\big)$ is non-zero since the vectors $\displaystyle \xi(0)=D_{\frac{d}{dt}}\left.\left(\psi_u\big(\gamma(t)\big)\right)\right|_{t=0}$ and $\psi_v\big(\gamma(0)\big)$ are linearly independent.

Since $\displaystyle D_{\frac{d}{dt}}\left.\left(\psi\big(\dot{\gamma}(t)\big)\right)\right|_{t=0}=a\xi(0)+2g(0)\psi_v\big(\gamma(0)\big)$, the equality \eqref{EqLimit} holds.
\end{proof}

\begin{prop}\label{PropTheSameLimit}
Under the assumptions of Proposition \ref{PropLimit}, if $\overline{\gamma}(t):=\gamma(-t)$, then
\begin{align}\label{EqLimit2}
\Psi_{\gamma}=\Psi_{\overline{\gamma}}.
\end{align}
\end{prop}
\begin{proof}
Since $\overline{\gamma}(t)=\gamma(-t)$, we get that $\dot{\overline{\gamma}}(t)=-\dot{\gamma}(-t)$ and in particular $\dot{\overline{\gamma}}(0)=-\dot{\gamma}(0)$. Since
\begin{align*}
D_\frac{d}{dt}\left(\psi\left(\dot{\overline{\gamma}}(t)\right)\right) &=
D_\frac{d}{dt}\left(-\psi\left(\dot{\gamma}(-t)\right)\right)=
-D_\frac{d}{dt}\left(\psi\left(\dot{\gamma}(-t)\right)\right)=
D_{-\frac{d}{dt}}\left(\psi\left(\dot{\gamma}(-t)\right)\right),
\end{align*}
the equality  \eqref{EqLimit2} holds.
\end{proof}

\begin{defn}
Let $\gamma_1$ and $\gamma_2$ be two $C^1$-regular curves emanating from $p$ such that $\mathcal E$-initial vectors of $\gamma_1$ and $\gamma_2$ at $p$ exist.   Then the angle
$$\arccos(\left<\Psi_{\gamma_1}, \Psi_{\gamma_2}\right>)\in [0,\pi]$$
 is called the \textit{angle between the initial vectors} of $\gamma_1$ and $\gamma_2$ at $p$.
 \end{defn}

We generalize the definition of singular sectors from \cite{SUY2} to the case of coherent tangent bundles over surfaces with boundary.

Let $U$ be a (sufficiently small) neighborhood of a singular point $p$. Let $\sigma_1$ and $\sigma_2$ be curves in $U$ starting at $p$ such that both are singular curves or one of them is a singular curve and the other one is in $\partial M$. A domain $\Omega$ is called a \textit{singular sector} at $p$ if it satisfies the following conditions
\begin{enumerate}[(i)]
\item the boundary of $\Omega\cap U$ consists of $\sigma_1$, $\sigma_2$ and the boundary of $U$.
\item $ \Omega \cap \Sigma=\emptyset$.
\end{enumerate}
If the peak $p\in M\setminus\partial M$ is an isolated singular point than the domain $U\setminus\{p\}$ is a singular sector at $p$, where $U$ is a neighborhood of $p$ such that $U\cap \Sigma=\{p\}$. We assume that singular direction is transversal to the boundary of $M$. Therefore there are no isolated singular points on the boundary.

We define the interior angle of a singular sector. If $p$ is in $\partial M$, then the \textit{interior angle} of a singular sector at $p$ is the angle of the initial vectors of $\sigma_1$ and $\sigma_2$ at $p$.

While the interior angle of a singular sector may take value greater than $\pi$ if $p\in M\setminus\partial M$, we can choose $\gamma_j$ for $j=0,\dots,n$ inside the singular sector   in a such way that the angel between $\Psi_{\gamma_{j-1}}$ and $\Psi_{\gamma_j}$ is not greater than $\pi$.

Let $\Omega$ be a singular sector at the peak $p$. Then there exists a positive integer $n$ and $C^1$-regular curves starting at $p$ $\gamma_0=\sigma_0,   \gamma_1, \cdots,  \gamma_n=\sigma_1$ satisfying the assumptions of Proposition \ref{PropPeakLimit} and the following conditions:
\begin{enumerate}[(i)]
\item if $i\neq j$ then $\gamma_i\cap\gamma_j=\emptyset$ in $\Omega$,\label{intersect}
\item for each $j=1,\ldots,n$ there exists  a sector domain $\omega_j\subset \Omega$ such that $\omega_j$ is bounded by $\gamma_{j-1}$ and $\gamma_j$ and $\omega_j\cap\gamma_i=\emptyset$ for $i\ne j-1, j$,
\item if $n\geqslant 2$ the vectors $\dot{\gamma}_{j-1}(0)$, $\dot{\gamma}_j(0))$  are linearly independent and form a positively oriented frame for $j=1,\ldots,n$. \label{basis}
\end{enumerate}
 If the peak $p$ is an isolated singular point then there exist curves $\gamma_0, \gamma_1, \gamma_2$ satisfying the above assumptions and conditions  (\ref{intersect})-(\ref{basis}). We also put $\gamma_3=\gamma_0$.

The \textit{interior angle of the singular sector} $\Omega$ is
\begin{align*}
 \sum_{j=1}^n \arccos(\left<\Psi_{\gamma_{j-1}},\Psi_{\gamma_j}\right>).
\end{align*}

If $\Omega$ is a singular sector at a singular point $p$ then $\Omega$ is contained in $M^{+}$ or $M^{-}$. The singular sector $\Omega$ is called \textit{positive}  (respectively \textit{negative}) if $\Omega\subset M^{+}$ (respectively $\Omega\subset M^{-}$ ).

\begin{defn}
Let $p$ be a singular point. Then $\alpha_+(p)$ (respectively $\alpha_-(p)$) is the sum of all interior angles of positive (respectively negative) singular sectors at $p$.
\end{defn}

\begin{prop}[Theorem A in \cite{SUY2}] Let $p\in M\setminus\partial M$ be a peak. The  sum $\alpha_+(p)$  of all interior angles of positive singular sectors at $p$ and the sum $\alpha_-(p)$  of all interior angles of  negative  singular sectors at $p$ satisfy
\begin{align*}
\alpha_+(p)+\alpha_-(p)&=2\pi, \\
\alpha_+(p)-\alpha_-(p)&\in\big\{-2\pi,0,2\pi\Big\}.
\end{align*}
\end{prop}

\begin{thm}\label{ThmAlphasNullA2}
Let $p\in\partial M$ be a singular point. We assume that  the singular direction is transversal to the boundary $\partial M$ at $p$.

If the null direction  is transversal to the boundary $\partial M$ at $p$, then
\begin{align*}
\alpha_+(p)+\alpha_-(p)&=\pi, \\
\alpha_+(p)-\alpha_-(p)&\in\big\{-\pi, \pi\Big\}.
\end{align*}

If the null direction is tangent to the boundary $\partial M$ at $p$, then
\begin{align*}
\alpha_+(p)=\alpha_-(p).
\end{align*}
\end{thm}
\begin{proof}
The first part of this theorem follows from Proposition 2.15 in \cite{SUY2}.
By Proposition \ref{PropTheSameLimit} we get the second part.
\end{proof}

\begin{defn}
A peak $p$ in $M\setminus\partial M$ is called \textit{positive} (\textit{null}, \textit{negative}, respectively) if $\alpha_+(p)-\alpha_-(p)>0$ ($\alpha_+(p)-\alpha_-(p)=0$, $\alpha_+(p)-\alpha_-(p)<0$, respectively).
\end{defn}

\begin{defn}
A singular point $p$ in $\partial M$ is called \textit{positive} (\textit{null}, \textit{negative}, respectively) if $\alpha_+(p)-\alpha_-(p)>0$ ($\alpha_+(p)-\alpha_-(p)=0$, $\alpha_+(p)-\alpha_-(p)<0$, respectively).
\end{defn}

 \begin{rem}
 It is easy to see that a peak $p$ in $\partial M$ is not null if $\partial M$ is transversal to the singular direction at $p$ and an $A_2$ singular point $p$ in $\partial M$ is null if the null vector at $p$ is tangent to $\partial M$.
 \end{rem}

\begin{defn}
Let $p$ be a peak in $\partial M$. We say that $p$ is in the \textit{positive boundary} (respectively in the \textit{negative boundary}) if there exists a neighborhood $U$ in $M$ of $p$ such that $(U\setminus\{p\})\cap\partial M\subset M^+$ (respectively $(U\setminus\{p\})\cap\partial M\subset M^-$).
\end{defn}

Let $\sigma(t)$ $( t\in (a;b) )$ be a $C^2$-regular curve on $M$.  We assume that if $\sigma(t)\in \Sigma$ then $\dot\sigma(t)$ is transversal to the null direction at $\sigma(t)$. Then  the image $\psi\left(\dot\sigma(t)\right)$ does not vanish. Thus we take a parameter $\tau$ of $\sigma$ such that
\begin{align*}
\left<\psi(\dot\sigma(\tau)),\psi(\dot\sigma(\tau))\right>\equiv 1,
\end{align*}
where $\displaystyle \dot\ =\frac{d}{d\tau}$.

\begin{defn}\label{DefGeodesic}
Let $n(\tau)$ be a section of $\mathcal{E}$ along $\sigma(\tau)$ such that $\{\psi(\dot\sigma(\tau)), n(\tau)\}$ is a positive orthonormal frame. Then
\begin{align*}
\hat{\kappa}_g(\tau):=\left<D_{\frac{d}{d\tau}}\psi(\dot\sigma(\tau)), n(\tau)\right>=\mu\left(\psi(\dot\sigma(\tau)),D_{\frac{d}{d\tau}}\psi(\dot\sigma(\tau))\right)
\end{align*}
is called the $\mathcal{E}$-{\it geodesic curvature} of $\sigma$, which gives the geodesic curvature of the  curve $\sigma$ with respect to the orientation of $\mathcal{E}$.
\end{defn}

We assume that  the curve $\sigma$ is a singular curve consisting of $A_2$-points.
Take a null vector field $\eta(\tau)$ along $\sigma(\tau)$ such that $\{\dot\sigma(\tau), \eta(\tau)\}$ is a positively oriented field along $\sigma(\tau)$ for each $\tau$. Then the \textit{singular curvature function} is defined by
\begin{align*}
\kappa_s(\tau):=\mbox{sgn}(d\lambda_{\psi} (\eta(\tau)))\cdot\hat{\kappa}_g(\tau),
\end{align*}
where $\mbox{sgn}(d\lambda_{\psi} (\eta(\tau)))$ denotes the sign of the function $d\lambda_{\psi} (\eta)$ at $\tau$. In a general parameterization of $\sigma=\sigma(t)$, the singular curvature function is defined as follows
\begin{align*}
\kappa_s(t)=\mbox{sgn}\left(d\lambda_{\psi} \left(\eta(t)\right)\right)\cdot\frac{\mu\left(\psi\left(\dot\sigma(t)\right), D_{\frac{d}{dt}}\psi\left(\dot\sigma(t)\right)\right)}{\left|\psi(\dot\sigma(t))\right|^3},
\end{align*}
where $\displaystyle \dot\ :=\frac{d}{dt}$, $\displaystyle |\xi|:=\sqrt{\left<\xi,\xi\right>}$.

By Proposition 1.7 in \cite{SUY2} the singular curvature function does not depend on the orientation of $M$, the orientation on $\mathcal E$, nor the parameter $t$ of the singular curve $\sigma(t)$.

By Proposition 2.11 in \cite{SUY2} the singular curvature measure $\kappa_sd\tau$ is bounded on any singular curve, where $d\tau$ is the arclength measure of this curve with respect to the first fundamental form $ds^2$. Now we prove the following proposition concerning the geodesic curvature measure on the boundary of $M$.

\begin{prop}\label{PropGeodesicCurvBound}
Let $\gamma:[0,\varepsilon)\to\partial M$ be a $C^2$-regular curve such that $\Sigma\cap\gamma([0,\varepsilon))=\{\gamma(0)\}$ is an $A_2$-point and the vector $\dot\gamma(0)$ is  the null vector at $\gamma(0)$. Then the geodesic curvature measure $\hat{\kappa}_gd\tau$ is continuous on $[0,\varepsilon)$ , where $d\tau$ is the arclength measure with respect to the first fundamental form $ds^2$.
\end{prop}
\begin{proof}

The point $\gamma(0)\in \partial M$ is a null $A_2$-point. By Proposition \ref{PropLimit} we can write that $\Psi(\dot{\gamma}(t))=t\zeta(t)$ for $t\in [0,\tilde \varepsilon)$ for sufficiently small $\tilde\varepsilon\le\varepsilon$, where $\zeta(t)\in \mathcal E_{\gamma(t)}$ and $\zeta(0)=D_{\frac{d}{dt}}\psi\left(\dot{\gamma}(t)\right)|_{t=0}\ne 0$. The geodesic curvature in  a general parameterization has the following form
\begin{align*}
\hat{\kappa}_g(t)=\frac{\mu\left(\psi\left(\dot{\gamma}(t)\right), D_{\frac{d}{dt}}\psi\left(\dot{\gamma}(t)\right)\right)}{\left|\psi(\dot{\gamma}(t))\right|^3}.
\end{align*}

Thus the geodesic curvature measure
\begin{align*}
\hat{\kappa}_g(\tau)d\tau=\hat{\kappa}_g(t)\left|\psi(\dot{\gamma}(t))\right|dt=\frac{\mu\left(\zeta(t), D_{\frac{d}{dt}}\zeta(t)\right)}{\left|\zeta(t)\right|^2}dt
\end{align*}
is bounded and continuous on $[0, \tilde \varepsilon)$. It implies that the geodesic curvature measure  is continuous on $[0,\varepsilon)$ since  $\Sigma\cap\gamma([0,\varepsilon))=\{\gamma(0)\}$.
\end{proof}

Let $U\subset M$ be a domain and let $\{e_1, e_2\}$ be a positive orthonormal frame field on $\mathcal{E}$ defined on $U$. Since $D$ is a metric connection,  there exists a unique $1$-form $\omega$ on $U$ such that
\begin{align*}
D_Xe_1=-\omega(X)e_2,\quad D_Xe_2=\omega(X)e_1,
\end{align*}
for any smooth vector field $X$ on $U$. The form $\omega$ is called the \textit{connection form} with respect to the frame $\{e_1, e_2\}$. It is easy to check that $d\omega$ does not depend on the choice of a frame $\{e_1, e_2\}$ and gives a globally defined $2$-form on $M$. Since $D$ is a metric connection and it satisfies (\ref{coherent}) we have
\begin{align*}
d\omega=Kd\hat{A}=\left\{\begin{array}{ll} KdA &\text{ on } M_+,\\ -KdA &\text{ on } M_-,\end{array}\right.
\end{align*}
where $K$ is the Gaussian curvature of the first fundamental form $ds^2$ (see \cite{SUY2, SUY}).

The next theorem is a generalization of the Gauss-Bonnet theorem for coherent tangent bundles over smooth compact oriented surfaces with boundary.

\begin{thm}[The Gauss-Bonnet type formulas]\label{ThmGaussBonnetFrontsWithB}

Let $\mathcal{E}$ be a coherent tangent bundle on a smooth compact oriented surface $M$ with boundary whose set of singular points $\Sigma$ admits at most peaks. If the set of singular points $\Sigma$ is transversal to the boundary $\partial M$, then

\begin{align}
\label{GBplusformula}2\pi\chi(M)&=\int_{M}KdA+2\int_{\Sigma}\kappa_sd\tau \\
\nonumber &+\int_{\partial M\cap M^+}\hat{\kappa}_gd\tau-\int_{\partial M\cap M^-}\hat{\kappa}_gd\tau-\sum_{p\in \text{null}(\Sigma\cap\partial M)}(2\alpha_+(p)-\pi),\\
\label{GBminusformula}\int_MKd\hat{A}&+\int_{\partial M}\hat{\kappa}_gd\tau = 2\pi\left(\chi(M^+)-\chi(M^-)\right)+2\pi\left(\# P^+-\# P^-\right)
\\ \nonumber &+\pi\left(\#(\Sigma\cap\partial M)^+ -\#(\Sigma\cap\partial M)^-\right)+\pi\left(\#P_{\partial M^+}-\#P_{\partial M^-}\right),
\end{align}
where $d\tau$ is the arc length measure, $P^+$ (respectively $P^-$) is the set of positive (respectively negative) peaks in $M\setminus\partial M$, $(\Sigma\cap\partial M)^+$ (respectively $(\Sigma\cap\partial M)^-$, $\text{null}(\Sigma\cap\partial M)$) is the set of positive (respectively negative, null) singular points in $\Sigma\cap\partial M$, $P_{\partial M^+}$ (respectively $P_{\partial M^-}$) is the set of peaks in the positive (respectively negative) boundary.

\end{thm}

\section{The proof of Theorem \ref{ThmGaussBonnetFrontsWithB}}\label{SectionProof}

We use the method presented in the proof of Theorem B in \cite{SUY2}.
First we formulate the local Gauss-Bonnet theorem for admissible triangles.

\begin{defn}\label{DefAdmissibleCurve}
A curve $\sigma(t)$ ($t\in[a,b]$) is \textit{admissible on the surface with boundary} if it satisfies one of the following conditions:
\begin{enumerate}
\item \label{Cond1Adm} $\sigma$ is a $C^2$ - regular curve such that $\sigma((a, b))$ does not contain a peak, and the tangent vector $\dot\sigma(t)$ ($t\in[a,b]$) is transversal to the singular direction, the null direction if $\sigma(t)\in\Sigma$ and $\dot{\sigma}(t)$ is transversal to the boundary if $\sigma(t)\in\partial M$.
\item \label{Cond2Adm} $\sigma$ is a $C^1$ - regular curve such that the set $\sigma([a,b])$ is contained in the set of singular points $\Sigma$ and the set $\sigma((a,b))$ does not contain a peak.
\item \label{Cond3Adm} $\sigma$ is $C^2$ - regular curve such that the set $\sigma([a,b])$ is contained in the boundary $\partial M$, the set $\sigma((a,b))$ does not contain a singular point and the tangent vector $\dot\sigma(t)$ ($t\in\{a,b\}$) is transversal to the singular direction if $\sigma(t)\in\Sigma$.
\end{enumerate}
\end{defn}

\begin{rem}
A curve $\sigma(t)$ is admissible in the sense of Definition 2.12 in \cite{SUY2} if it satisfies conditions \eqref{Cond1Adm} or \eqref{Cond2Adm} in Definition \ref{DefAdmissibleCurve}. For the purpose of this paper we add \eqref{Cond3Adm} in Definition \ref{DefAdmissibleCurve} and the transversality of the admissible curve to the boundary in \eqref{Cond1Adm}.
\end{rem}

Let $U$ be a domain in $M$.
\begin{defn}[see Definition 3.1 in \cite{SUY2}]\label{DefAdmTriangle}
Let $\overline{T}\subset U$ be the closure of a simply connected domain $T$ which is bounded by three admissible arcs $\gamma_1$, $\gamma_2$, $\gamma_3$. Let $A$, $B$ and $C$ be the distinct three boundary points of $T$ which are intersections of these three arcs. Then $\overline{T}$ is called an \textit{admissible triangle on the  surface with boundary} if it satisfies the following conditions:
\begin{enumerate}
\item $\overline{T}$ admits at most one peak on $\{A, B, C\}$.
\item the three interior angles at $A$, $B$ and $C$ with respect to the metric $g$ are all less than $\pi$.
\item if $\gamma_j$ for $j=1, 2, 3$ is not a singular curve, it is $C^2$-regular, namely it is a restriction of a certain open $C^2$-regular arc.
\end{enumerate}
\end{defn}

We write $\Delta ABC:=\overline{T}$ and we denote by
\begin{align*}
\overline{BC}:=\gamma_1,\ \overline{CA}:=\gamma_2,\ \overline{AB}:=\gamma_3
\end{align*}
the regular arcs whose boundary points are $\{B,C\}, \{C,A\}, \{A,B\}$, respectively.

We give the orientation of $\partial\Delta ABC$ compatible with respect to the orientation of $M$.
We  denote by $\angle A, \angle B, \angle C$ the interior angles (with respect to the first fundamental form $ds^2$) of the piecewise smooth boundary of $\Delta ABC$ at $A, B$ and $C$, respectively if $A, B$ and $C$ are regular points.

If $A\in M\setminus\partial M$ is a singular point and $(U;u,v)$ is a $g$-coordinate system at $A$, then we set (see Proposition 2.15 in \cite{SUY2})
\begin{align*}
\angle A:=\left\{\begin{array}{ll}
\pi & \text{if the } u-\text{curve passing through } A \text{ separates } \overline{AB} \text{ and } \overline{AC},\\
0 & \text{otherwise}.
\end{array}\right.
\end{align*}

Let $\sigma(t)$ be an admissible curve. We define a \textit{geometric curvature} $\tilde{\kappa}_g(t)$ in the following way:
\begin{align*}
\tilde{\kappa}_g(t)=\left\{\begin{array}{rll} \hat\kappa_g(t) & \text{if} & \sigma(t)\in M^+, \\
-\hat\kappa_g(t) & \text{if} & \sigma(t)\in M^-,\\ \kappa_s(t) & \text{it} & \sigma(t)\in\Sigma,\end{array}\right.
\end{align*}
where $\hat\kappa_g$ is the geodesic curvature with respect to the orientation of $M$ and $\kappa_s$ is the singular curvature.
\begin{prop}\label{PropAdmTriangleOur}
Let $\Delta ABC$ be an admissible triangle on the surface with boundary such that $A$ is an $A_2$-point, $\overarc{AB}\subset\partial M$ and $\Delta ABC\setminus\overarc{AC}$ lies in $M^+$ or in $M^-$. Suppose that the boundary $\partial M$ is transversal to $\Sigma$ at $A$ and let $T_A\partial M$ be a null direction at $A$. Then
\begin{align}\label{EqLocalGBOurCase}
\angle A+\angle B+\angle C-\pi=\int_{\partial\Delta ABC}\tilde{\kappa}_gd\tau+\int_{\Delta ABC}K dA.
\end{align}
\end{prop}
\begin{proof}
Without loss of generality, let us assume that $\Delta ABC\setminus\overarc{AC}$ lies in $M^+$. If the arc $\overarc{AC}\subset\Sigma$ or the interior angle $\angle BAC$ with respect to the metric $g$ is greater than $\frac{\pi}{2}$, we decompose the triangle $\Delta ABC$ into admissible triangles $\Delta ABD$ and $\Delta ADC$ such that the interior angle $\angle BAD$ with respect to the metric $g$ is in the interval $(0, \frac{\pi}{2})$ and the arc $\overarc{AD}$ is transversal to the arc $\overarc{BC}$ at $D$, see Fig. \ref{FigTriangles}. The formula \eqref{EqLocalGBOurCase} for $\Delta ADC$ follows from Theorem 3.3 in \cite{SUY2}, so it is enough to prove the formula \eqref{EqLocalGBOurCase} for the triangle $\Delta ABD$.

\begin{figure}[h]
\centering
\includegraphics[scale=0.18]{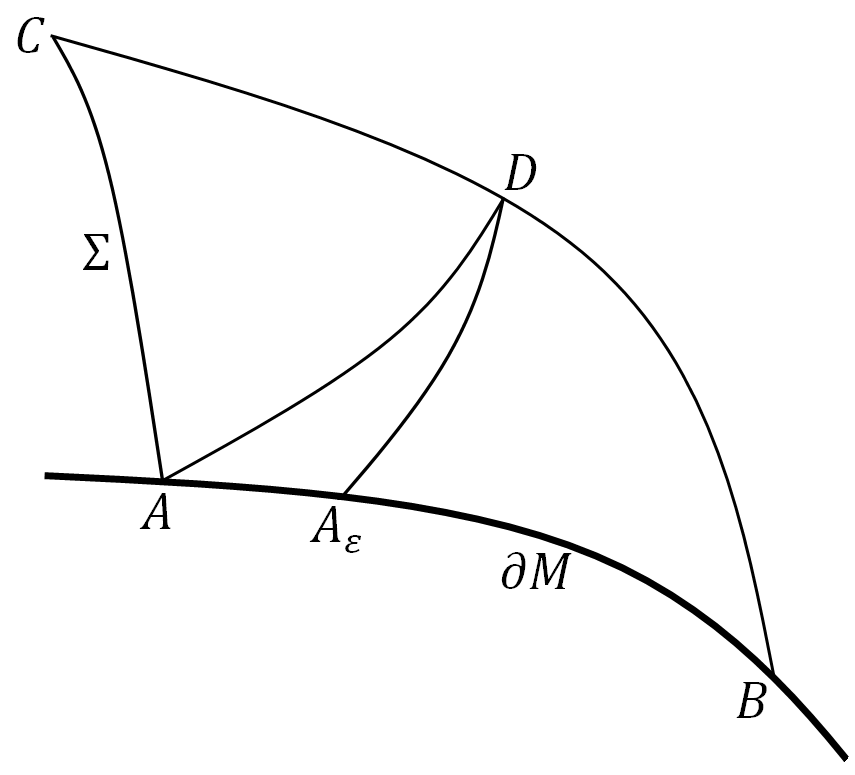}
\caption{}
\label{FigTriangles}
\end{figure}

We can take the arc $\overarc{AD}$ and rotate it around $D$ with respect to the canonical metric $du^2+dv^2$ on the $uv$-plane. Then we obtain a smooth one-parameter family of $C^2$-regular arcs starting at $D$. Since the interior angle $\angle BAD$ is in $(0, \frac{\pi}{2})$ and $\overarc{BD}$, $\overarc{AD}$ are transversal at $D$, restricting the image of this family to the triangle $\Delta ABD$, we obtain a family of $C^2$-regular curves
\begin{align*}
\gamma_{\varepsilon}:[0,1]\to\Delta ABD,
\end{align*}
where $\varepsilon\in[0,1]$ and:
\begin{enumerate}[(i)]
\item $\gamma_0$ parameterizes $\overline{AD}$ and $\gamma_0(0)=A$, $\gamma_0(1)=D$,
\item $\gamma_{\varepsilon}(1)=D$ for all $\varepsilon\in[0,1]$,
\item the correspondence $\sigma:\varepsilon\mapsto\gamma_{\varepsilon}(0)$ gives a subarc of $\overarc{AB}$. We set $A_{\varepsilon}=\gamma_\varepsilon(0)$, where $A_0=A$.
\end{enumerate}

Since $\Delta A_{\varepsilon}BD$ for $\varepsilon > 0$ is an admissible triangle, then by Theorem 3.3 in \cite{SUY2} we get that
\begin{align*}
\angle A_{\varepsilon}+\angle B+\angle A_{\varepsilon}DB-\pi=\int_{\partial\Delta A_{\varepsilon}BD}\tilde{\kappa}_gd\tau+\int_{\Delta A_{\varepsilon}BD}K dA.
\end{align*}

Since $\Delta ABD$ is admissible and $\tilde{\kappa}_g$ is bounded on both $\overarc{AB}$ and $\overarc{AD}$, by taking the limit as $\varepsilon\to 0^+$, we have that
\begin{align*}
\lim_{\varepsilon\to 0^+}\angle A_{\varepsilon}+\angle B+\angle D-\pi=\int_{\partial\Delta ABD}\tilde{\kappa}_gd\tau+\int_{\Delta ABD}K dA.
\end{align*}

By Proposition \ref{PropLimit} we have
\begin{align}\label{EqAngleA}
\lim_{\varepsilon\to 0^+}\angle A_{\varepsilon} &=
\lim_{\varepsilon\to 0^+}\arccos\frac{\left<\psi\left(\frac{d\gamma_{\varepsilon}(t)}{dt}\right)\big|_{t=0},\psi\left(\frac{d\sigma(\varepsilon)}{d\varepsilon}\right)\right>}{\left|\psi\left(\frac{d\gamma_{\varepsilon}(t)}{dt}\right)\big|_{t=0}\right|\cdot\left|\psi\left(\frac{d\sigma(\varepsilon)}{d\varepsilon}\right)\right|}=\arccos\left<\Psi_{\gamma_0}, \Psi_{\sigma}\right>.
\end{align}
This completes the proof.
\end{proof}

\begin{rem}
By Theorem 3.3 in \cite{SUY2} and Proposition \ref{PropAdmTriangleOur} the equality \eqref{EqLocalGBOurCase} holds for any admissible triangle on a surface with boundary.
\end{rem}

Let $\overline{X}$, $X^\circ$, $\partial X$, respectively,  denote the closure of a subset $X$ of $M$, the interior of $X$ and the boundary of $X$, respectively.

Let us triangulate $M$ by admissible triangles such that each point in the set $P\cup(\Sigma\cap\partial M)=:P^\star$ is a vertex, where $P$ is the set all peaks in $M^\circ$. Let $T$, $E$ and $V$, respectively,  denote the set of all triangles, the set of all edges and the set all of vertices in the given triangulation, respectively.

\begin{lem}\label{LemRelationGB}
The following relation holds:
\begin{align*}
&\no{v\in V\ |\ v\in (M^+)^\circ}=\chi(\overline{M}^+)+\frac{1}{2}\no{\Delta\in T\ |\ \Delta\subset\overline{M}^+}+\\
&\qquad+\frac{1}{2}\no{e\in E\ |\ e\subset\partial M^+}-\no{v\in V\ |\ v\in\partial M^+\setminus P^\star}-\#P^\star.
\end{align*}
\end{lem}
\begin{proof}
By the definition of Euler's characteristic we get that
\begin{align}\label{EqEulerCh}
\no{v\in V\ |\ v\in \overline{M}^+}=\chi(\overline{M}^+)-\no{\Delta\in T\ |\ \Delta\subset \overline{M}^+}+\no{e\in E\ |\ e\subset \overline{M}^+}.
\end{align}
Furthermore, it is easy to verify that
\begin{align}\label{EqRelBetweenFandE}
\no{e\in E\ |\ e\subset \overline{M}^+}=\frac{3}{2}\no{\Delta\in F\ |\ \Delta\subset \overline{M}^+}+\frac{1}{2}\no{e\in E\ |\ e\subset\partial M^+}
\end{align}
and
\begin{align}\label{EqVertices}
\no{v\in V\ |\ v\in (M^+)^\circ}=&\no{v\in V\ |\ v\in \overline{M}^+}\\ \nonumber &-\no{v\in V\ |\ v\in\partial M^+\setminus P^\star}-\no{p\in V\ |\ p\in P^\star}.
\end{align}

Combining together \eqref{EqEulerCh}, \eqref{EqRelBetweenFandE} and \eqref{EqVertices} we end the proof.
\end{proof}

Let us define the sum $\displaystyle\sum_{\Delta ABC\in T, \Delta\subset\overline{M}^+}(\angle A+\angle B+\angle C-\pi)$ by $S_+$.

Then

\begin{align*}
S_+=&2\pi\no{v\in V\ |\  v\in (M^+)^\circ}+\pi \no{v\in V\ |\ v\in\partial M^+\setminus P^\star} \\ &+\sum_{p\in P^\star}\alpha_+(p)-\pi\no{\Delta\in T\ |\ \Delta\subset \overline{M}^+}.
\end{align*}

By Lemma \ref{LemRelationGB} we get that

\begin{align*}
S_+&=2\pi\chi(\overline{M}^+)+\pi\no{e\in E\ |\ e\in\partial M^+}-\pi\no{v\in V\ |\ v\in\partial M^+\setminus P^{\star}}\\ &\ \ \ -2\pi\# P^{\star}+\sum_{p\in P^{\star}}\alpha_+(p)\\
&=2\pi\chi(\overline{M}^+)+\frac{\pi}{2}\sum_{v\in V, v\in\partial M^+}\deg_{\partial M^+}(v)-\pi\no{v\in V\ |\ v\in\partial M^+}\\ &\ \ \ -\pi\# P^{\star}+\sum_{p\in P^{\star}}\alpha_+(p),
\end{align*}
where $\deg_{X}(v)=\no{e\in E\ |\ e\subset X, v\in e}$, where $X$ is a subset of $M$. Since $\partial M^+$ is an Eulerian graph, the number $\deg_{\partial M^+}(v)$ is even and let us write that \linebreak $\displaystyle m_+(v):=\frac{1}{2}\deg_{\partial M^+}(v)$. Furthermore, if $v\in (V\cap\partial M^+)\setminus P^\star$ then $\deg_{\partial M^+}(v)=2$ and we get the relation
\begin{align*}
\frac{1}{2}\sum_{v\in V\cap\partial M^+}\deg_{\partial M^+}(v)-\no{v\in V\ |\ v\in\partial M^+}=\sum_{p\in P^\star}(m_+(p)-1).
\end{align*}
Hence we get the following:
\begin{align}\label{EqSPlus}
S_+=2\pi\chi(\overline{M}^+)+\sum_{p\in P^\star}(\alpha_+(p)+\pi m_+(p))-2\pi\# P^\star.
\end{align}
Similarly we get that
\begin{align}\label{EqSMinus}
S_-=2\pi\chi(\overline{M}^-)+\sum_{p\in P^\star}(\alpha_-(p)+\pi m_-(p))-2\pi\# P^\star,
\end{align}
where $\displaystyle S_-=\sum_{\Delta ABC\in T, \Delta\subset\overline{M}^-}(\angle A+\angle B+\angle C-\pi)$ and $\displaystyle m_-(v):=\frac{1}{2}\mbox{deg}_{\partial M^-}(v)$.

It is easy to see that
\begin{align}
\label{peaksboundaryeq1} & m_+(p)=m_-(p)\text{ for }p\in P^\star\setminus\partial M,\\
\label{peaksboundaryeq2} & m_+(p)+m_-(p)=\deg_\Sigma(p)\text{ for }p\in P^\star\setminus\partial M,\\
\label{peaksboundaryeq3} & m_+(p)+m_-(p)=\deg_{\Sigma\cup\partial M}(p)-1\text{ for }p\in\Sigma\cap\partial M.\end{align}

Furthermore if $p\in\Sigma\cap\partial M$, then

\begin{align}
\label{peaksboundaryeq4} m_+(p)-m_-(p)=\left\{\begin{array}{ll} 1 & \text{if } p \text{ is a peak in the positive boundary},\\
-1 & \text{if }  p \text{ is a peak in the negative boundary},\\
0 & \text{otherwise.}\end{array}\right.
\end{align}

\begin{lem}\label{LemEulerSigma}
The Euler characteristic of $\Sigma$ is equal to
\begin{align*}
\chi(\Sigma)=\# P^\star-\frac{1}{2}\sum_{p\in P^\star}\left(m_+(p)+m_-(p)\right)+\frac{1}{2}\#(\Sigma\cap\partial M).
\end{align*}
\end{lem}
\begin{proof}
We know that
\begin{align*}
\chi(\Sigma) &= \no{v\in V\ |\ v\in\Sigma}-\no{e\in E\ |\ e\subset\Sigma}\\
&=\no{v\in V\ |\ v\in\Sigma}-\frac{1}{2}\sum_{v\in V\cap\Sigma}\deg_\Sigma v.
\end{align*}
If $p\in P\setminus\partial M$ then $\deg_\Sigma(p)=\deg_{\Sigma\cup\partial M}(p)$ and if $p\in\Sigma\cap\partial M$ then $\deg_{\Sigma}(p)=\deg_{\Sigma\cup\partial M}(p)-2$. By \eqref{peaksboundaryeq2} and \eqref{peaksboundaryeq3} we get that
\begin{align*}
\chi(\Sigma)&=\# P^\star-\frac{1}{2}\sum_{p\in P\setminus\partial M}(m_+(p)+m_-(p))-\frac{1}{2}\sum_{p\in\Sigma\cap\partial M}(m_+(p)+m_-(p)-1)\\
&=\# P^\star-\frac{1}{2}\sum_{p\in P^\star}\left(m_+(p)+m_-(p)\right)+\frac{1}{2}\#(\Sigma\cap\partial M).
\end{align*}
\end{proof}

\begin{lem}\label{LemSumSS}
The following equality holds:
\begin{align*}
S_++S_-=2\pi\chi(M)+\sum_{p \in \text{null}(\Sigma\cap\partial M)}(2\alpha_+(p)-\pi).
\end{align*}
\end{lem}
\begin{proof}
Since $\chi(\overline{M}^+)+\chi(\overline{M}^-)=\chi(M)+\chi(\Sigma)$, by \eqref{EqSPlus}, \eqref{EqSMinus}, Lemma \ref{LemEulerSigma} and Theorem \ref{ThmAlphasNullA2} we get that:
\begin{align*}
S_++S_-& =2\pi\chi(M)+2\pi\chi(\Sigma)+\sum_{p\in P^\star}(\alpha_+(p)+\alpha_-(p))\\ &\ \ \ +\pi\sum_{p\in P^\star}(m_+(p)+m_-(p))-4\pi\# P^\star\\
&=2\pi\chi(M)+\pi\#(\Sigma\cap\partial M)+\sum_{p\in P^\star}(\alpha_+(p)+\alpha_-(p))-2\pi\# P^\star\\
&=2\pi\chi(M)+\pi\#(\Sigma\cap\partial M)+\sum_{p \in (\Sigma\cap\partial M)^+\cup (\Sigma\cap\partial M)^-}(\alpha_+(p)+\alpha_-(p))\\
& \ \ \ +\sum_{p\in P\setminus\partial M}(\alpha_+(p)+\alpha_-(p))+\sum_{p \in \text{null}(\Sigma\cap\partial M)}(\alpha_+(p)+\alpha_-(p))\\
&\ \ \ -2\pi\#(\Sigma\cap\partial M)-2\pi\#(P\setminus\partial M)\\
&=2\pi\chi(M)+\sum_{p \in (\Sigma\cap\partial M)^+\cup (\Sigma\cap\partial M)^-}\pi+\sum_{p\in P\setminus\partial M}2\pi\\ &\ \ \  +\sum_{p \in \text{null}(\Sigma\cap\partial M)}2\alpha_+(p)-\pi\#(\Sigma\cap\partial M)-2\pi\#(P\setminus\partial M)\\
&=2\pi\chi(M)+\sum_{p \in \text{null}(\Sigma\cap\partial M)}(2\alpha_+(p)-\pi).
\end{align*}
\end{proof}

\begin{lem}\label{LemMinusSS}
The following equality holds:
\begin{align*}
S_+-S_-&=2\pi\left(\chi(M^+)-\chi(M^-)\right)+2\pi\left(\# P^+-\# P^-\right) \\ &+\pi\left(\#(\Sigma\cap\partial M)^+-\#(\Sigma\cap\partial M)^-\right)+\pi\left(\#P_{\partial M^+}-\#P_{\partial M^-}\right),
\end{align*}
where $P^+$ (respectively $P^-$) is the set of positive (respectively negative) peaks in $M\setminus\partial M$, $(\Sigma\cap\partial M)^+$ (respectively $(\Sigma\cap\partial M)^-$) is the set of positive (respectively negative) singular points in $\Sigma\cap\partial M$, $P_{\partial M^+}$ (respectively $P_{\partial M^-}$) is the set of peaks in the positive (respectively negative) boundary.
\end{lem}
\begin{proof}
It is a consequence of \eqref{EqSPlus}, \eqref{EqSMinus}, Lemma \ref{LemEulerSigma} and Theorem \ref{ThmAlphasNullA2} and the fact that $\chi(\overline{M}^+)-\chi(\overline{M}^-)=\chi(M^+)-\chi(M^-)$.
\end{proof}

Since the integration of the geometric curvature on curves which are not included in  $\Sigma\cup \partial M$ are canceled by opposite integrations and the singular curvature does not depend on the orientation of the singular curve, by Proposition \ref{PropAdmTriangleOur} and Theorem 3.3 in \cite{SUY2} we get that

\begin{align*}
S_{\pm}=\int_{M^{\pm}}KdA+\int_{\partial M^{\pm}}\tilde{\kappa}_gd\tau=\int_{M^{\pm}}KdA+\int_{\Sigma}\kappa_{s}d\tau\pm\int_{\partial M\cap M^{\pm}}\hat{\kappa}_gd\tau.
\end{align*}

Hence

\begin{align}
\label{EqSumSS} S_++S_-&=\int_MKdA+2\int_\Sigma\kappa_sd\tau+\int_{\partial M\cap M^+}\hat{\kappa}_gd\tau-\int_{\partial M\cap M^-}\hat{\kappa}_gd\tau,\\
\label{EqMinusSS} S_+-S_-&=\int_MKd\hat{A}+\int_{\partial M}\hat{\kappa}_gd\tau.
\end{align}

By Lemma \ref{LemSumSS}, Lemma \ref{LemMinusSS}, \eqref{EqSumSS} and \eqref{EqMinusSS} we complete the proof of Theorem \ref{ThmGaussBonnetFrontsWithB}.

\section{Applications of the Gauss-Bonnet formulas to maps}\label{GenApplications}

As a corollary of Theorem \ref{ThmGaussBonnetFrontsWithB} we get Fukuda-Ishikawa's theorem \cite{FI} (see also \cite{KS1}), which is the generalization of Quine's formula (\cite{Q1}) for surfaces with boundary (see also Proposition 3.6. in \cite{SUY3}).

\begin{prop}
Let $M$ and $N$ both be compact oriented connected surfaces with boundary. Let $f:M\to N$ be a $C^{\infty}$-map such that $f(\partial M)\subset\partial N$ and $f^{-1}(\partial N)=\partial M$ and whose set of singular points consists of folds and cusps. If the set of singular points of $f$ is transversal to $\partial M$ then the topological degree of $f$ satisfies
\begin{align}\label{QuineTypeFormula}
\mbox{deg}(f)\chi(N)=\chi(M_f^+)-\chi(M_f^-)+S_f^+-S_f^-,
\end{align}
where $M^+_f$ (respectively $M^-_f$) is the set of regular points at which $f$ preserves (respectively reverses) the orientation, $S^+_f$ (respectively $S^-_f$) is the number of positive cusps (respectively the number of negative cusps).
\end{prop}
\begin{proof}
Let $h$ be a Riemannian metric on $N$ and let $D$ be the Levi-Civita connection on $(N, h)$. Then the tuple $(f^{\ast}TN, h, D, df)$ is a coherent tangent bundle on $M$ (see \cite{SUY3}). Since $f(\partial M)\subset\partial N$ and the set of singular points of $f$ is transversal to $\partial M$, there are no cusps in $\partial M$ and all folds in $\partial M$ are null singular points. Therefore by Theorem \ref{ThmGaussBonnetFrontsWithB} we get that:
\begin{align}\label{EqQuineTypeFormula1}
&\int_MKd\hat{A}+\int_{\partial M}\hat{\kappa}_gd\tau= 2\pi\left(\chi(M^+_f)-\chi(M^-_f)\right)+2\pi\left(S_f^+-S_f^-\right).
\end{align}

The following identity holds
\begin{align*}
\int_{M}Kd\hat A &=\int_{M}f^*\Omega_{12},
\end{align*}
where $\Omega_{12}$ is a curvature $2$-form.

Furthermore, it is well known that $\displaystyle\int_{M}f^*\Omega_{12}=\mbox{deg}(f)\int_{N}\Omega_{12}$ (see for instance Remark $1$ in \cite{DFN1} page $111$). On the other hand, we have $\displaystyle\int_N\Omega_{12}=\int_NK_NdA$, where $K_N$ is the Gaussian curvature of $N$. By the Gauss-Bonnet theorem for $N$ we get $\displaystyle\int_{N}K_NdA =2\pi\chi(N)-\int_{\partial N}\kappa_{g}d\tau$, where $\kappa_g$ is the geodesic curvature of $\partial N$ in $N$. Thus

\begin{align}\label{EqQuineTypeFormula2}
\int_{M}Kd\hat A=\mbox{deg}(f)\left(2\pi\chi(N)-\int_{\partial N}\kappa_gd\tau\right).
\end{align}

Since $f(\partial M)\subset\partial N$ and $f^{-1}(\partial N)=\partial M$ and $\left<\cdot, \cdot\right>_p=h_{f(p)}(\cdot, \cdot )$ for $p$ in $M$, we obtain that
\begin{align}\label{EqQuineTypeFormula3}
\displaystyle\int_{\partial M}\hat{\kappa}_gd\tau=\mbox{deg}\left(f\big|_{\partial M}\right)\int_{\partial N}\kappa_gd\tau.
\end{align}

By Theorem 13.2.1 (\cite{DFN1} page 105) we get $\mbox{deg}(f)=\mbox{deg}\left(f\big|_{\partial M}\right)$.

By \eqref{EqQuineTypeFormula1}-\eqref{EqQuineTypeFormula3} we obtain the formula \eqref{QuineTypeFormula}.
\end{proof}

We can also get easily the generalization of Proposition 3.7. in \cite{SUY3} by the Gauss-Bonnet formulas.

\begin{prop}
Let $(N, h)$ be an oriented Riemannian $2$-manifold with boundary, let $M$ be a compact oriented $2$-manifold with boundary. Let $f:M\to N$ be a $C^{\infty}$-map such that $f(\partial M)\subset\partial N$ and whose set of singular points consists of folds and cusps, the set of singular points of $f$ is transversal to $\partial M$. Then the total singular curvature $\displaystyle\int_{\Sigma}\kappa_s d\tau$ with respect to the length element $d\tau$ (with respect to $h$) on the set of singular points $\Sigma$ is bounded, and satisfies the following identity
\begin{align*}
2\pi\chi(M)&=\int_{M}\left(\widetilde{K}\circ f\right)\left|f^*dA_h\right|+2\int_{\Sigma}\kappa_sd\tau \\
\nonumber &+\int_{\partial M\cap M^+_f}\hat{\kappa}_gd\tau-\int_{\partial M\cap M^-_f}\hat{\kappa}_gd\tau-\sum_{p \in \text{null}(\Sigma\cap\partial M)}(2\alpha_+(p)-\pi).
\end{align*}
where $M^+_f$ (respectively $M^-_f$) is the set of regular points at which $f$ preserves \linebreak (respectively reverses) the orientation,  $\widetilde{K}$ is the Gaussian curvature function on $(N,h)$,  $\hat{\kappa} _g$ is a geodesic curvature, $|f^*dA_h|$ is the pull-back of the Riemannian measure of $(N,h)$ and
\begin{align*}
\alpha_+(p)=\mbox{arccos}\left(h\left(\frac{D_{\frac{d}{dt}}\left(\frac{d}{dt}\left(f\circ\gamma\right)(t)\right)}{\left|D_{\frac{d}{dt}}\left(\frac{d}{dt}\left(f\circ\gamma\right)(t)\right)\right|}, \frac{d}{d\tau}(f\circ\sigma)(\tau)\right)\right),
\end{align*}
where $D$ is the Levi - Civita connection on $N$, $\gamma$ is a $C^2$ -  parameterization of the boundary $\partial M$ in the neighborhood of $p$ and $\sigma$ is a parameterization of $\Sigma$ in the neighborhood of $p$.
\end{prop}

\section{Geometry of the affine extended wave front}\label{SectionBigFront}

In this section we apply Theorem \ref{ThmGaussBonnetFrontsWithB} to an affine extended wave front of a planar non-singular hedgehog. Fronts are examples of coherent tangent bundles (see \cite{SUY2}).

Planar hedgehogs are curves which can be parameterized using their Gauss map. A hedgehog can be also viewed as the Minkowski difference of convex bodies (see \cite{LLR1, MM1, MM2, MM3, MM4}). The non-singular hedgehogs are also known as the rosettes (see \cite{CM1, MMo, Zw3}).

The singularities and the geometry of affine $\lambda$-equidistants were very widely studied in many papers \cite{B1, DJRR1, DMR1, DRS1, GWZ1, JJR1, RT1, Z1}. The envelope of affine diameters (the Centre Symmetry Set) was studied in \cite{DR1, GH1, GR1, GZ1, J1}.

Let $\C$ be a smooth parameterized curve on the affine plane $\mathbb{R}^2$, i.e. the image of the $C^{\infty}$ smooth map from an interval to $\mathbb{R}^2$. We say that a smooth curve is \textit{closed} if it is the image of a $C^{\infty}$ smooth map from $S^1$ to $\mathbb{R}^2$. A smooth curve is \textit{regular} if its velocity does not vanish. A regular curve is called an $m$-rosette if its signed curvature is positive and its rotation number is $m$. A \textit{convex} curve is a $1$-rosette.

\begin{defn}
A pair of points $a,b\in\C$ ($a\neq b$) is called a \textit{parallel pair} if the tangent lines to $\C$ at $a$ and $b$ are parallel.
\end{defn}

\begin{defn}
An \textit{affine $\lambda$-equidistant} is the following set:
\begin{align*}
E_{\lambda}(\C)=\Big\{\lambda a+(1-\lambda)b\ \Big|\ a,b\text{ is a parallel pair of }\C\Big\}.
\end{align*}
The set $E_{\frac{1}{2}}(\C)$ will be called the \textit{Wigner caustic} of $\C$.
\end{defn}

\begin{defn}
The \textit{Centre Symmetry Set} of $\C$, which we will denote as $\Css(\C)$, is the envelope of all chords passing through parallel pairs of $\C$.
\end{defn}

 If $\C$ is a generic convex curve, then  the Wigner caustic of $\C$, $\Eq_{\lambda}(\C)$, for a generic $\lambda$, and $\Css(\C)$ are smooth closed curves with at most cusp singularities (\cite{B1, GH1, GZ1, J1}), the number of cusps of the Wigner caustic and the Centre Symmetry Set of $\C$ are odd and not smaller than $3$ (\cite{B1, GH1}), the number of cusps of $\Css(\C)$ is not smaller than the number of cusps of $\Eq_{\frac{1}{2}}(\C)$ (\cite{DR1}) and the number of cusps of $\Eq_{\lambda}(\C)$ is even for a generic $\displaystyle\lambda\neq\frac{1}{2}$ (\cite{DZ1}). Moreover, cusp singularities of all $\Eq_{\lambda}(\C)$ are lying on smooth parts of $\Css(\C)$ (\cite{GZ1}). In addition, if $\C$ is a convex curve, then the Wigner caustic is contained in a closure of the region bounded by the Centre Symmetry Set (\cite{C1}, see Fig. \ref{FigOvalWcCssEq}). The Wigner caustic also appears in one of the two constructions of bi-dimensional improper affine spheres. This construction can be generalized to higher even dimensions (\cite{CDR1}). The oriented area of the Wigner caustic improves the classical planar isoperimetric inequality and gives the relation between the area and the perimeter of smooth convex bodies of constant width (\cite{Zw1, Zw2, Zw3}). Recently the properties of the middle hedgehog, which is a generalization of the Wigner caustic in the case of non-smooth convex bodies, were studied in \cite{S2, S3}.

\begin{figure}[h]
\centering
\includegraphics[scale=0.25]{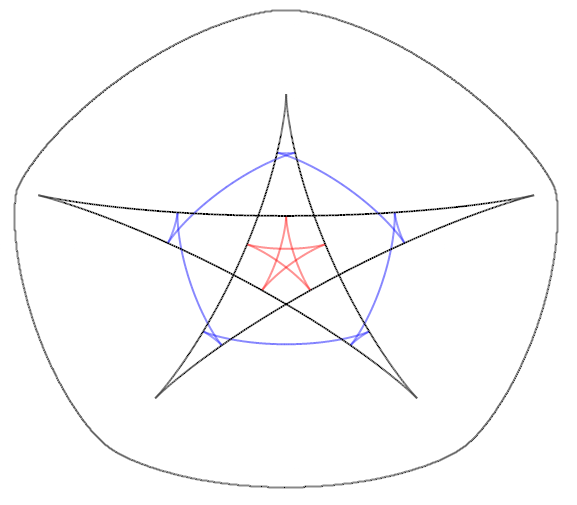}
\caption{An oval $\C$ and $\Eq_{0.5}(\C)$, $\Eq_{0.4}(\C)$, $\Css(\C)$. The support function of $\C$ is equal to $p(\theta)=31+2\cos 2\theta+\sin 5\theta$.}
\label{FigOvalWcCssEq}
\end{figure}

\begin{defn} \textit{The extended affine space} is the space $\mathbb{R}^{3}_e=\mathbb{R}\times\mathbb{R}^2$ with coordinate $\lambda\in\mathbb{R}$ (called the affine time) on the first factor and a projection on the second factor denoted by $\pi:\mathbb{R}^3_e\ni(\lambda,x)\mapsto x\in\mathbb{R}^2$.
\end{defn}

\begin{defn}\label{DefExtWaveFront}
Let $R_m$ be an $m$-rosette. \textit{The affine extended wave front} of $R_m$, $\E(R_m)$, is the union of all $\Eq_{\lambda}(R_m)$ for $\lambda\in[0,1]$, each embedded into its own slice of the extended affine space
\begin{align*}
\E(R_m)=\bigcup_{\lambda\in [0,1]}\{\lambda\}\times\Eq_{\lambda}(R_m)\subset\mathbb{R}^3_e.
\end{align*}
\end{defn}

Note that, when $R_m$ is a circle on the plane, then $\E(R_m)$ is the double cone, which is a smooth manifold with the nonsingular projection $\pi$ everywhere, but at its singular point, which projects to the center of the circle (the center of symmetry).

We will study the geometry of $\E(R_m)$ through the support function of $R_m$ (\cite{CM1, Zw3}). Take a point $O$ as the origin of our frame. Let $\theta$ be the oriented angle from the positive $x_1$-axis. Let $p(\theta)$ be the oriented perpendicular distance from $O$ to the tangent line at a point on $R_m$ and let this ray and $x_1$-axis form an angle $\theta$. The function $p$ is a single valued periodic function of $\theta$ with period $2m\pi$ and the parameterization of $R_m$ in terms of $\theta$ and $p(\theta)$ is as follows
\begin{align}\label{ParameterRm}
[0,2m\pi)\ni\theta\mapsto\gamma(\theta)=\big(p(\theta)\cos\theta-p'(\theta)\sin\theta, p(\theta)\sin\theta+p'(\theta)\cos\theta\big)\in\mathbb{R}^2.
\end{align}
Then, the radius of curvature $\rho$ of $R_m$ is in the following form
\begin{align}
\rho(\theta)=\frac{ds}{d\theta}=p(\theta)+p''(\theta)>0,
\end{align}
or equivalently, the curvature $\kappa$ of $R_m$ is given by
\begin{align}\label{CurvatureRm}
\kappa(\theta)=\frac{d\theta}{ds}=\frac{1}{p(\theta)+p''(\theta)}>0.
\end{align}

In Fig. \ref{PictureBigFront} we illustrate (with different opacities) the surface $\E(R_1)$, where $R_1$ is an oval {represented by the support function $p(\theta)=11-0.5\cos 2\theta+\sin 3\theta$. We also present the following curves: $\{0\}\times R_1$, $\{1\}\times R_1$, $\{0.5\}\times E_{0.5}(R_1)$,  $\{0\}\times E_{0.5}(R_1)$ and $\{0\}\times\Css(R_1)$.

\begin{figure}[h]
\centering
\includegraphics[scale=0.2]{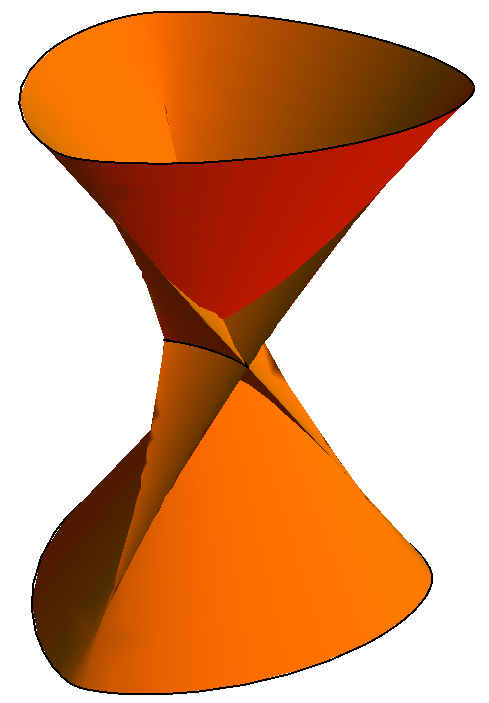}
\includegraphics[scale=0.2]{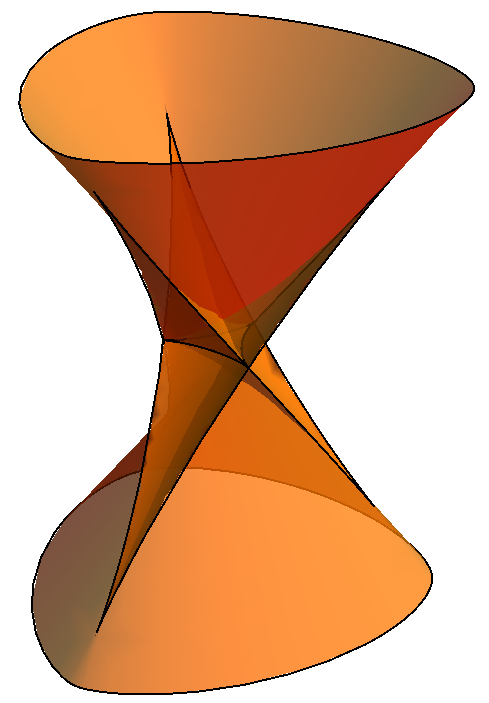}
\includegraphics[scale=0.2]{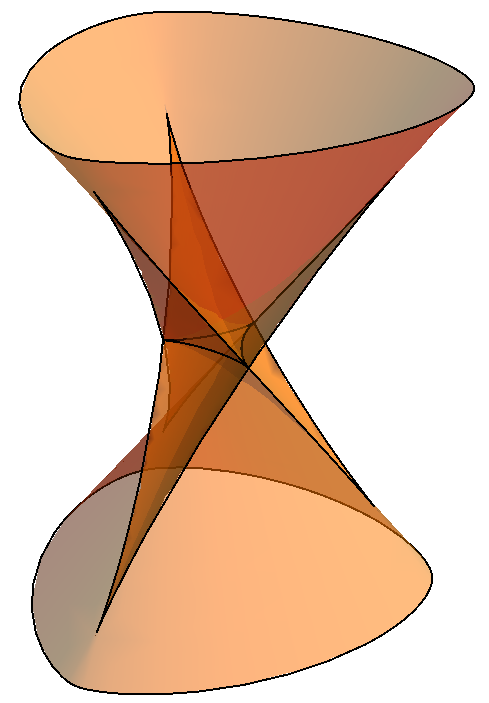}
\includegraphics[scale=0.2]{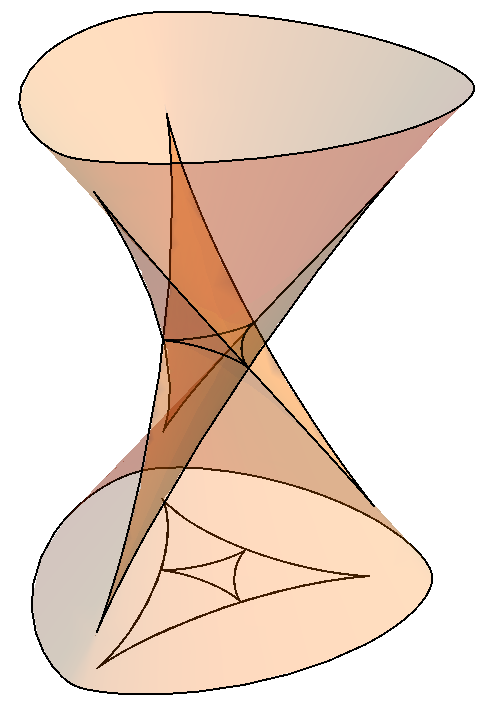}
\caption{}
\label{PictureBigFront}
\end{figure}

Let $\Sigma$ be a set of singular points of $\E$. It is well known that $\pi(\Sigma(\E(R_1)))=\Css(R_1)$ and the map $\Sigma(\E(R_1))\ni p\mapsto \pi(p)\in\Css(R_1)$ is the double covering of $\Css(R_1)$.

\begin{rem}\label{RemKnownAboutBranches}
In \cite{DZ1, Zw3} we study in details the geometry of affine $\lambda$-equidistants of rosettes. We show among other things that there exist $m$ branches of $E_{0.5}(R_m)$ and $2m-1$ branches of $E_{\lambda}(R_m)$ for $\lambda\neq 0, 0.5, 1$. Let $E_{0.5, k}(R_m)$ for $k=1, 2,\ldots, m$ denote different branches of $E_{0.5}(R_m)$ and let $E_{\lambda, k}(R_m)$ for $k=1, 2, \ldots, 2m-1$ denote different branches of $E_{\lambda}(R_m)$ for $\lambda\neq 0, 0.5, 1$. Then the support function of $E_{0.5, k}(R_m)$ for $k=1,\ldots, m$ is in the form (\ref{SupportWCk}), the support function of $E_{\lambda, k}(R_m)$ for $k=1, 2,\ldots, m$ (respectively $k=m+1, m+2, \ldots, 2m-1$) in the form (\ref{SupportEqk1}) (respectively in the form (\ref{SupportEqk2})), where
\begin{align}
\label{SupportWCk} p_{0.5, k}(\theta) &= \frac{1}{2}\big(p(\theta)+(-1)^kp(\theta+k\pi)\big),\\
\label{SupportEqk1} p_{\lambda, k}(\theta) &=\lambda p(\theta)+(-1)^k(1-\lambda)p(\theta+k\pi),\\
\label{SupportEqk2} p_{\lambda, k}(\theta) &=(1-\lambda)p(\theta)+(-1)^{k}\lambda p(\theta+(k-m)\pi).
\end{align}
Let $\gamma_{\lambda, k}$ denote the parameterization of $E_{\lambda, k}$ in terms of the support function accordingly to (\ref{SupportWCk}), (\ref{SupportEqk1}) and (\ref{SupportEqk2}), respectively.
Furthermore each branch of $E_{\lambda}(R_m)$, except $E_{0.5, m}(R_m)$, has the rotation number equal to $m$. The rotation number of $E_{0.5, m}(R_m)$ is equal to $\displaystyle\frac{m}{2}$.
If $R_m$ is a generic $m$-rosette then for $\lambda\in(0,1)-\{0.5\}$ only branches $E_{\lambda, k}(R_m)$ for $k=1, 3, \ldots, 2\lceil 0.5m\rceil -1, m+1, m+3, \ldots, m+2\lceil 0.5m\rceil -1$ can admit cusp singularities and branches $E_{0.5, k}(R_m)$ for $k=1, 3, \ldots, 2\lceil 0.5m\rceil -1$ has cusp singularities. By \cite{GH1} we known that if $a,b$ is parallel pair of $R_m$ and $R_m$ is parameterized at $a$ and $b$ in different directions and $\kappa(a), \kappa(b)$ denote the signed curvatures of $R_m$ at $a$ and $b$, respectively, then the point $\displaystyle\frac{a\kappa_1+b\kappa_2}{\kappa_1+\kappa_2}$ which is lying on the line between $a$ and $b$, belongs to $\Css(R_m)$.
\end{rem}

\begin{cor}
Let $R_m$ be a generic $m$-rosette. Then $\Css(R_m)$ which is created from singular points of $E_{\lambda}(R_m)$ for $\lambda\in[0,1]$ consists of exactly $2\lceil 0.5m\rceil-1$ branches.
\end{cor}
\begin{proof}
It is a consequence of Remark \ref{RemKnownAboutBranches}.
\end{proof}

Let $\Css_k(R_m)$ for $k=1, 3, \ldots, 2\lceil 0.5m\rceil -1$ denote a branch of $\Css(R_m)$. Then the parameterization of $\Css_k(R_m)$ is in the following form
\begin{align}\label{CsskParameter}
\gamma_{\Css_k(R_m)}(\theta)=\frac{\kappa(\theta)}{\kappa(\theta)+\kappa(\theta+k\pi)}\gamma(\theta)+\frac{\kappa(\theta+k\pi)}{\kappa(\theta)+\kappa(\theta+k\pi)}\gamma(\theta+k\pi),
\end{align}
where if $k<m$ then $\theta\in[0,2m\pi]$ and if $k=m$ then $\theta\in[0,m\pi]$.

\begin{lem}\label{LemCusps}
Let $\C$ be a closed smooth curve with at most cusp singularities and let the rotation number of $\C$ be $m$. If $m$ is an integer, then the number of cusp singularities is even. If $m$ is the form $0.5d$, where $d$ is an odd integer, then the number of cusp singularities is odd.
\end{lem}
\begin{proof}
\begin{figure}[h]
\centering
\includegraphics[scale=0.23]{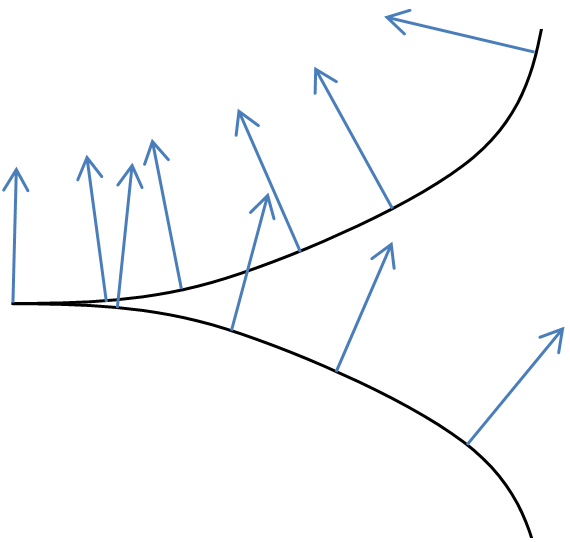}
\caption{}
\label{FigCuspVectorField}
\end{figure}
A continuous normal vector field to the germ of a curve with the cusp singularity is directed outside the cusp on the one of two connected regular components and is directed inside the cusp on the other component as it is showed in Fig. \ref{FigCuspVectorField}. If $m$ is an integer, then the number of cusps of $\C$ is even, otherwise is odd.
\end{proof}

\begin{prop}
Let $R_m$ be a generic $m$-rosette. If $k=m$ and $m$ is an odd number, then the number of cusp singularities of $\Css_k(R_m)$ is odd and not smaller than the number of cusp singularities of $E_{0.5, k}(R_m)$, otherwise the number of cusp singularities of $\Css_k(R_m)$ is even and not smaller than the number of cusp singularities of $E_{0.5, k}(R_m)$, which is even and positive.
\end{prop}
\begin{proof}
The parity of the number of cusp singularities of $\Css_k(R_m)$ is a consequence of (\ref{CsskParameter}) and Lemma \ref{LemCusps}.

Let $m$ be even and $k\leqslant m$ or $m$ be odd and $k<m$. By Theorem 2.9 in \cite{Zw3} we know that $E_{0.5, k}(R_m)$ has at least $2$ cusp singularities. Because the cusp in $E_{0.5}$ appears when $\displaystyle\frac{\kappa(a)}{\kappa(b)}=1$ and cusp in $\Css$ appears when $\displaystyle\left(\frac{\kappa(a)}{\kappa(b)}\right)'=0$ (\cite{DR1, GH1}), where $a,b$ is a parallel pair and $'$ is used to denote the derivative with respect to the parameter along the corresponding segment of a curve. Therefore by Roll's theorem we get that the number of cusp singularities of $\Css_k(R_m)$ is not smaller than the number of cusp singularities of $E_{0.5, k}(R_m)$. The same arguments works when $m$ is odd and $k=m$.
\end{proof}

Let $\E_k(R_m)$ for $k=1, \ldots, m$ be a branch of $\E(R_m)$ which has the following parameterization
\begin{align}\label{BranchBigFrontParameterization}
f_k(\lambda, \theta)=\left(\lambda, \lambda\gamma(\theta)+(1-\lambda)\gamma(\theta+k\pi)\right).
\end{align}

We use the following notation: 
\begin{align}
\label{FrontDerNotation}
\left(f_k\right)_{\lambda}:=\frac{\partial}{\partial\lambda}f_k(\lambda, \theta), \  \ \left(f_k\right)_{\theta}:=\frac{\partial}{\partial\theta}f_k(\lambda,\theta).
\end{align}

In Fig. \ref{PictureBigFrontRosette21} and Fig. \ref{PictureBigFrontRosette22} we illustrate (with different opacities) the branches $\E_1(R_2)$ and $\E_2(R_2)$, respectively, where $R_2$ is a $2$ - rosette represented by the support function $p(\theta)=11+\sin\frac{\theta}{2}-7\cos\frac{3\theta}{2}-\frac{1}{2}\sin 2\theta$.

\begin{figure}[h]
\centering
\includegraphics[scale=0.2]{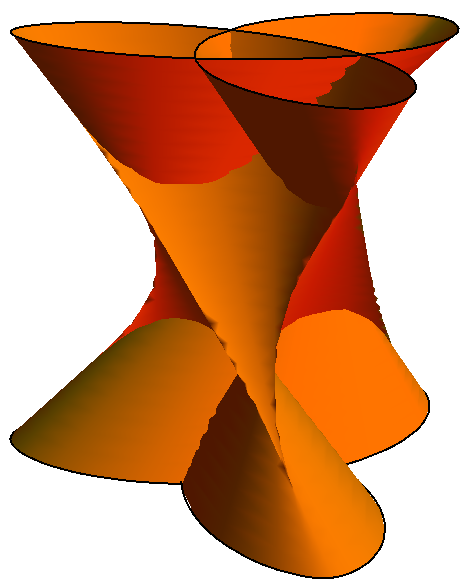}
\includegraphics[scale=0.2]{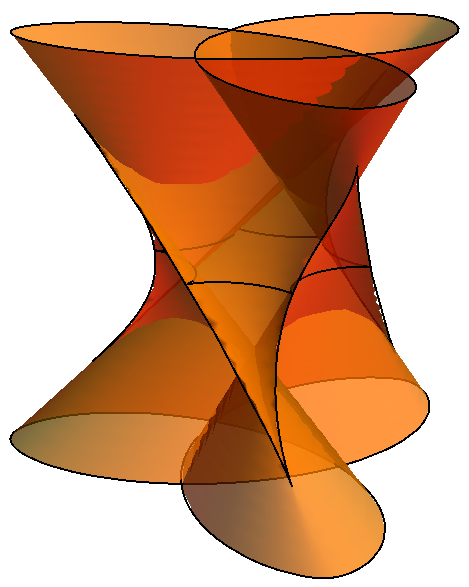}
\includegraphics[scale=0.2]{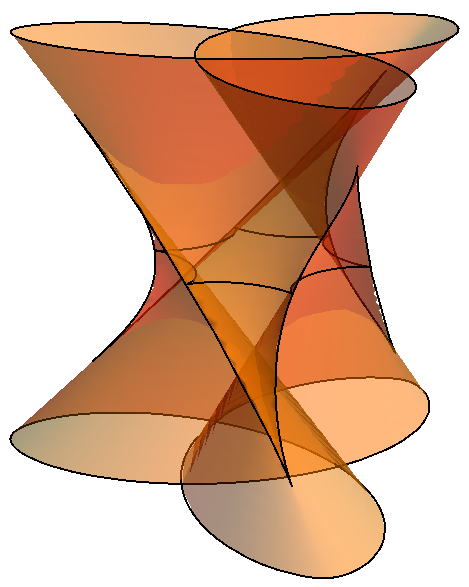}
\includegraphics[scale=0.2]{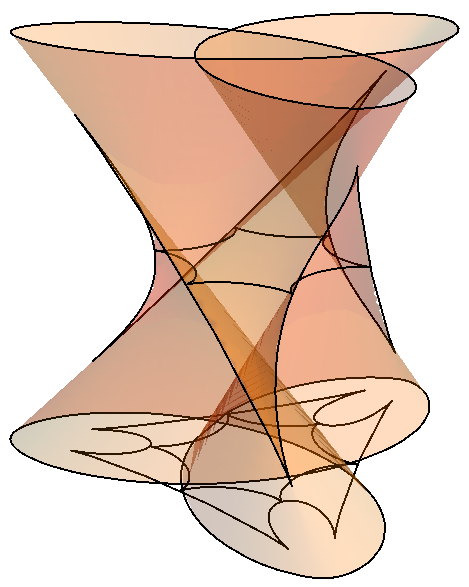}
\caption{}
\label{PictureBigFrontRosette21}
\end{figure}

\begin{figure}[h]
\centering
\includegraphics[scale=0.2]{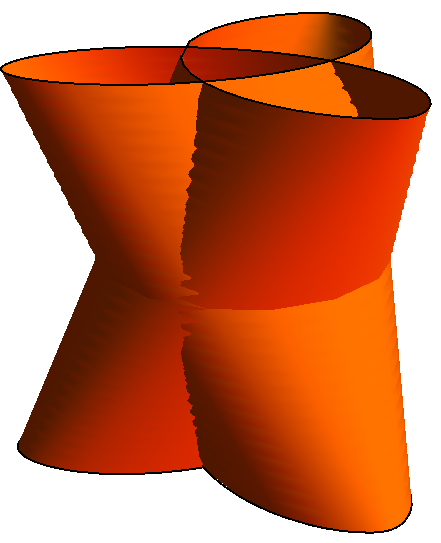}
\includegraphics[scale=0.2]{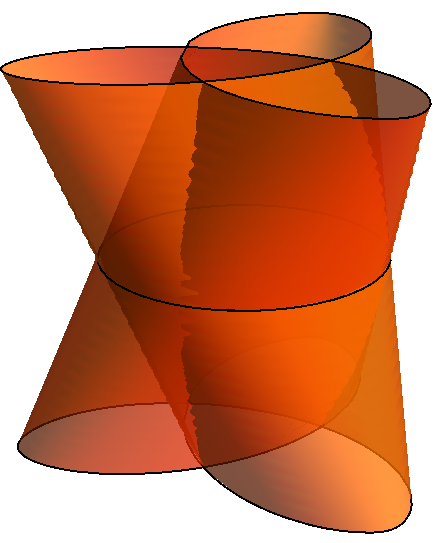}
\includegraphics[scale=0.2]{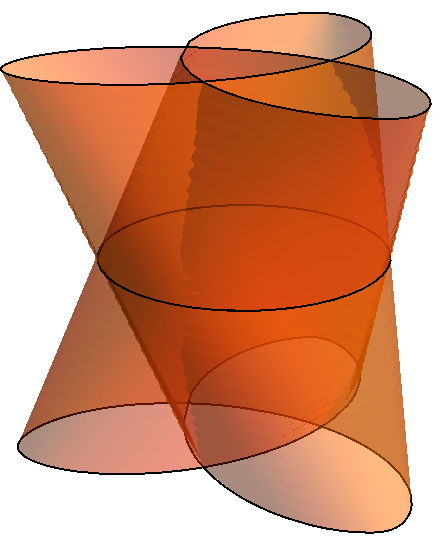}
\includegraphics[scale=0.2]{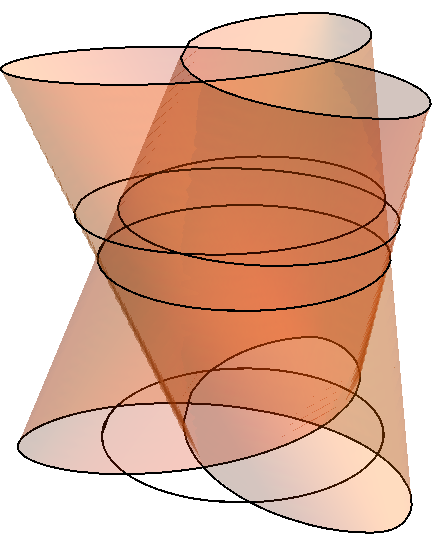}
\caption{}
\label{PictureBigFrontRosette22}
\end{figure}

Directly by Definition \ref{DefExtWaveFront} we get the following proposition.

\begin{prop}\label{BigFrontRuledSurface}
Every branch of $\E(R_m)$ is a ruled surface.
\end{prop}

It is well known that the Gaussian curvature of a ruled surface at a non-singular point is non-positive. By direct calculation we get the following proposition.

\begin{prop}\label{BigFrontGaussCurv}
Let $R_m$ be an $m$ - rosette.
\begin{enumerate}[(i)]
\item A point $(\lambda,\theta)$ is a singular point of $\E_k(R_m)$ if and only if
\begin{align}\label{ConditionESingular}
\frac{\kappa(\theta)}{\kappa(\theta+k\pi)}=(-1)^{k+1}\frac{\lambda}{1-\lambda}.
\end{align}
\item A singular point $(\lambda_0,\theta_0)$ is a cuspidal edge if and only if
\begin{align}\label{CondCuspEdgeThm}
\left(\frac{\kappa (\theta+k\pi)}{\kappa(\theta)}\right)'\Big|_{(\lambda_0,\theta_0)}\neq 0.
\end{align}
\item A singular point $(\lambda_0,\theta_0)$ is a swallowtail if and only if
\begin{align}\label{CondSwallowTailThm}
\left(\frac{\kappa(\theta+k\pi)}{\kappa(\theta)}\right)'\Big|_{(\lambda_0,\theta_0)}=0\ \text{and}\ \left(\frac{\kappa(\theta+k\pi)}{\kappa(\theta)}\right)''\Big|_{(\lambda_0,\theta_0)}\neq 0.
\end{align}
\item If $R_m$ is generic then every singular point of $\E(R_m)$ is non-degenerate.
\end{enumerate}

\end{prop}
\begin{proof}
We use (\ref{BranchBigFrontParameterization}) as the parameterization of $\E_k(R_m)$. Let us notice that $f_k$ is singular if and only if $\left(f_k\right)_{\lambda}\times\left(f_k\right)_{\theta}=0$. This condition is equivalent to (\ref{ConditionESingular}). By Fact 1.5 in \cite{SUY} we get \eqref{CondCuspEdgeThm} and \eqref{CondSwallowTailThm}. By direct calculation we get (iv) (see Definition \ref{SignedAreaDef}).
\end{proof}

\begin{cor}
Let $R_m$ be an $m$-rosette. Then the number of branches of $\E(R_m)$ is equal to $m$ and a branch $\E_k(R_m)$ is singular if and only if $k$ is odd.
\end{cor}

In Fig. \ref{PictureBigFrontRosette21} and in Fig. \ref{PictureBigFrontRosette22} we present two branches of $\E(R_2)$: $\E_1(R_2)$ and $\E_2(R_2)$, respectively.

\begin{prop}\label{BigFrontGaussCurv}
Let $R_m$ be an $m$ - rosette and let $p$ be a non-singular point of $\E_k(R_m)$. Then the Gaussian curvature of $\E_k(R_m)$ at $p$ is equal to $0$.
\end{prop}
\begin{proof}
The surface is parameterized by (\ref{BranchBigFrontParameterization}).

At a non-singular point $(\lambda, \theta)$ the Gaussian curvature $K$ of $\E_k$ is equal to
{\footnotesize
\begin{align}
K_k&(\lambda,\theta)=\\ \nonumber
&\frac{\det\left(\left(f_k\right)_{\lambda\lambda},\left(f_k\right)_{\lambda},\left(f_k\right)_{\theta}\right) \cdot \det\left(\left(f_k\right)_{\theta\theta},\left(f_k\right)_{\lambda},\left(f_k\right)_{\theta}\right)-\det^2\left(\left(f_k\right)_{\lambda \theta},\left(f_k\right)_{\lambda},\left(f_k\right)_{\theta}\right)}{ \left(|\left(f_k\right)_{\lambda}|^2|\left(f_k\right)_{\theta}|^2-(\left(f_k\right)_{\lambda}\cdot\left(f_k\right)_{\theta})^2\right)^2}.
\end{align}
}

Since $\left(f_k\right)_{\lambda\lambda}=0$ and vectors $\left(f_k\right)_{\theta}$ and $\left(f_k\right)_{\lambda\theta}$ are linearly dependent, the Gaussian curvature $K_k$ at a non-singular point of $\E_k$ is equal to zero.
\end{proof}

\begin{defn}
Let $R_m$ be an $m$-rosette. Let $k\in\{1,2,\ldots,m\}$. Then the $k$-\textit{width} of $R_m$ for an oriented angle $\theta$ is the following
\begin{align}\label{EqWidth}
w_k(\theta)=p(\theta)-(-1)^kp(\theta+k\pi).
\end{align}
\end{defn}

\begin{figure}[h]
\centering
\includegraphics[scale=0.46]{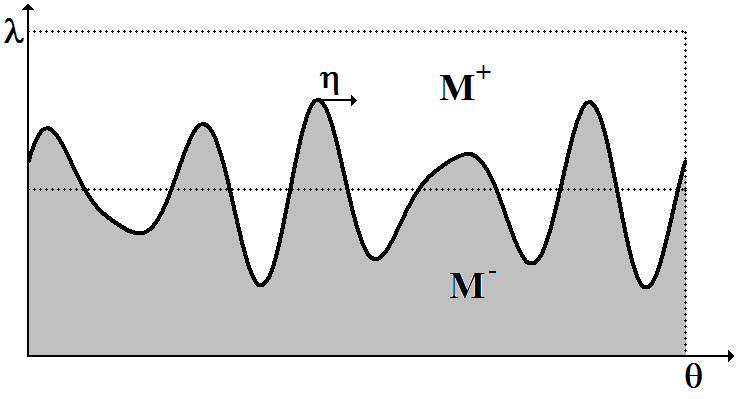}
\caption{}
\label{FigNPlusNMinus}
\end{figure}

\begin{rem}\label{WaveFrontIsFront}
Let $R_m$ be a generic $m$ - rosette and $k\leqslant m$ be an odd number. From now on we set
\begin{align*}
&M:=[0,1]\times S^1,\\
&M\ni (\lambda,\theta)\mapsto f_k(\lambda,\theta)\in\E_k(R_m)\subset\mathbb{R}^3,\\
&M\ni(\lambda,\theta)\mapsto\nu_k(\lambda,\theta):=\frac{(w_k(\theta), \mathbbm{n}(\theta))}{\sqrt{1+w_k^2(\theta)}}\in S^2.
\end{align*}
The map $(f_k,\nu_k)$ is a front. Then the coherent tangent bundle $\mathcal{E}^{f_k}$ over $M$ has the following fiber at $p\in M$
\begin{align*}
\mathcal{E}^{f_k}_p:=\left\{X\in T_{f_k(p)}\mathbb{R}^3\ |\ \left<X, \nu_k(p)\right>=0\right\}.
\end{align*}
The set of singular points $\Sigma_k$ is parameterized by $(\lambda_k(\theta),\theta)$, where \linebreak $\displaystyle\lambda_k(\theta)=\frac{\kappa(\theta)}{\kappa(\theta)+\kappa(\theta+k\pi)}$. Let us notice that
\begin{align*}
M^-=\left\{(\lambda,\theta)\in M\ \big|\ \lambda<\lambda_k(\theta)\right\}, M^+=\left\{(\lambda,\theta)\in M\ \big|\ \lambda>\lambda_k(\theta)\right\}.
\end{align*}
Furthermore, if the function $\lambda_k(\theta)$ has a local minimum, then the point  $(\lambda_k(\theta),\theta)$ is a negative peak and if $\lambda_k(\theta)$ has a local maximum, then  this point is a positive peak. See Fig. \ref{FigNPlusNMinus}.
\end{rem}

\begin{prop}\label{BigFrontGeodesicCurv}
Let $R_m$ be a generic $m$-rosette. Let $k$ be an odd number and let $\lambda\in[0,1]$. Then the $\mathcal{E}^{f_k}$ - geodesic curvature of a curve $\{\lambda\}\times S^1$ in $M$ at a non-singular point is equal to
\begin{align}\label{GeoCurvForm}
\hat\kappa_{k, g}(\theta):=\frac{w_k(\theta)}{|\lambda\rho(\theta)-(1-\lambda)\rho(\theta+k\pi)|\sqrt{1+w_k^2(\theta)}}.
\end{align}
\end{prop}
\begin{proof}
Let $s_k(\lambda,\theta):=\lambda\rho(\theta)-(1-\lambda)\rho(\theta+k\pi)$. Then \eqref{GeoCurvForm} follows from the formula
\begin{align*}
\hat\kappa_{k, g}(\theta)=
\frac{\det(\gamma'_{k,\lambda}(\theta), \gamma''_{k,\lambda}(\theta), \nu_k(\lambda,\theta))}{|\gamma'_{k,\lambda}(\theta)|^3}.
\end{align*}
\end{proof}

\begin{prop}
Let $R_m$ be a generic $m$-rosette. Let $k$ be an odd number. Then the singular curvature on a cuspidal edge at a point\linebreak $\displaystyle\left(\frac{\kappa(\theta)}{\kappa(\theta)+\kappa(\theta+k\pi)},\ \theta\right)$ is equal to
\begin{align}\label{CssKsRelation}
\kappa_{k,s}(\theta)=\kappa_{\Css_k}(\theta)\cdot \frac{\sqrt{1+w_k^2(\theta)}}{w_k(\theta)}\cdot\left(\frac{w_k^2(\theta)+w_k'^2(\theta)}{1+w_k^2(\theta)+w_k'^2(\theta)}\right)^{\frac{3}{2}},
\end{align}
where $\kappa_{\Css_k}(\theta)$ is a the curvature of $\Css_k(R_m)$, which is given by the following formula:
\begin{align}\label{KCssParameterizationTheta}
&\kappa_{\Css_k}(\theta)=\frac{-\big(\kappa(\theta)+\kappa(\theta+k\pi)\big)^3}{\kappa(\theta)\kappa(\theta+k\pi)\big|\kappa'(\theta+k\pi)\kappa(\theta)-\kappa'(\theta)\kappa(\theta+k\pi)\big|}\cdot\frac{w_k(\theta)}{\big(w_k^2(\theta)+w_k'^2(\theta)\big)^{\frac{3}{2}}}.
\end{align}
\end{prop}
\begin{proof}
It is a direct consequence of the formula of the singular curvature and the formula of the curvature of the Centre Symmetry Set (see Lemma 2.6 in \cite{DZ1}).
\end{proof}

\begin{figure}[h]
\centering
\includegraphics[scale=0.3]{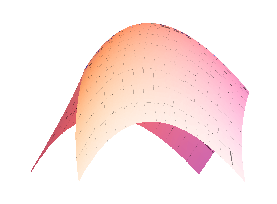}
\hspace{0.5cm}
\includegraphics[scale=0.3]{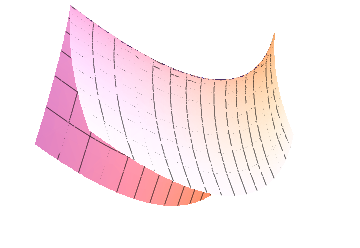}
\caption{Examples of positively (on the left) and negatively (on the right) curved cuspidal edges.}
\label{FigurePosNegCE}
\end{figure}

By Theorem 1.6 in \cite{SUY} we know that the singular curvature does not depend on the orientation of the parameter $\theta$, the orientation
of $M$, the choice of $\nu$, nor the orientation of the singular curve. The sign of the singular curvature have a geometric interpretation, if the singular curvature is positive (respectively negative) then the cuspidal edge is positively (respectively negatively) curved. See Fig. \ref{FigurePosNegCE}.

We find a formula which gives us the relation between the total singular curvature on set of singular points and the total geodesic curvature on the boundary of $M$. The integrals in \eqref{GBExtWF}-\eqref{CutInHalfMinus} can be seen as integrals on $f_k(\Sigma_k)$ and $f_k(\{\lambda\}\times S^1)=\{\lambda\}\times E_{k,\lambda}(R_m)$ since the arclength measure, the singular curvature and $\mathcal{E}^{f_k}$ - geodesic curvature are defined with respect to the first fundamental form $ds^2$ which is the pullback of metric $\left<\cdot, \cdot\right>$ on $\E_k(R_m)\subset\mathbb{R}^3$.

\begin{thm}\label{ThmTotalGeodSing}
Let $k$ be an odd number. Let $R_m$ be a generic $m$-rosette. Then
\begin{align}\label{GBExtWF}
\int_{\Sigma_k}\kappa_{k, s} d\tau+\int_{\{1\}\times S^1}\hat\kappa_{k, g}d\tau=0,
\end{align}
where $d\tau$ denote the arc length measure and the orientation of $\{1\}\times S^1$ is compatible with the orientation of $M$.
\end{thm}
\begin{proof}

By Remark \ref{WaveFrontIsFront} we get that $(f_k, \nu_k): M\to \mathbb{R}^3\times S^2$ is a front. The boundary of $M$ does not intersect the set of singular points $\Sigma$. By genericity of $R_m$ this front satisfies the assumptions of Theorem \ref{ThmGaussBonnetFrontsWithB}. Since $\lambda_k(\theta)+\lambda_k(\theta+k\pi)=1$, we get that $M^+$ and $M^-$ are symmetric. Hence $\chi(M^+)=\chi(M^-)$ and $\#P^-=\#P^+$.

By Proposition \ref{BigFrontGaussCurv} and Theorem  \ref{ThmGaussBonnetFrontsWithB} we get that
\begin{align*}
\int_{\{1\}\times S^1}\hat\kappa_{k,g}d\tau=-\int_{\{0\}\times S^1}\hat\kappa_{k,g}d\tau
\end{align*}
and then we get (\ref{GBExtWF}).
\end{proof}

\begin{thm}
Let $k$ be an odd number, $R_m$ be a generic $m$-rosette and $\lambda\in [0,1)$. If $E_{k, \lambda}(R_m)$ admits at most cusp singularities, then
\begin{align}
\label{LambdaGeodesic}
&\int_{\{\lambda\}\times S^1}\hat\kappa_{k,g}d\tau=-\int_{\{1\}\times S^1}\hat\kappa_{k,g}d\tau,\\
\label{CutInHalfPlus}
&\int_{(\{\frac{1}{2}\}\times S^1)\cap M^+}\hat\kappa_{k,g}d\tau =\sum_{p\in C}\alpha_+(p)-\frac{1}{2}\pi\#C-\frac{1}{2}\int_{\{1\}\times S^1}\hat\kappa_{k,g}d\tau,\\
\label{CutInHalfMinus}
&\int_{(\{\frac{1}{2}\}\times S^1)\cap M^-}\hat\kappa_{k,g}d\tau =-\sum_{p\in C}\alpha_+(p)+\frac{1}{2}\pi\#C-\frac{1}{2}\int_{\{1\}\times S^1}\hat\kappa_{k,g}d\tau,
\end{align}
where the orientations of $S^1$ in the integrals on the left hand sides and the right hand sides are opposite in the above formulas, $C=\Sigma_k\cap(\{\frac{1}{2}\}\times S^1)$, $d\tau$ is the arclength measure and
\begin{align}\label{AngleFrontFormula}
\alpha_+(p):=\mbox{arccos}\left(\sqrt{\frac{w_k^2(\theta)+w'^2_k(\theta)}{1+w^2_k(\theta)+w'^2_k(\theta)}}\cos(\beta(\theta))\right),
\end{align}
where $p=(\frac{1}{2}, \theta)$ and $\beta(\theta)$ is the angle between the tangent vector to $R_m$ at $\gamma(\theta)$ and the vector $\gamma(\theta+k\pi)-\gamma(\theta)$.

\end{thm}
\begin{proof}
Let $M_{\lambda}:=[\lambda,1]\times S^1$. By Remark \ref{WaveFrontIsFront} we get that $(f_k, \nu_k)\big|_{M_{\lambda}}:M_{\lambda}\to\mathbb{R}^3\times S^2$ is a front. It is easy to see that $\chi(M_{\lambda}^+)=0$ and $\chi(M_{\lambda}^{-})=\# P^+-\# P^-$ is the number of cusps of $\mathcal{E}^{f_k}\big|_{M_{\lambda}}$ (that is $\#\big(\Sigma_k\cap(\{\lambda\}\times S^1)\big)$). Since every point $p\in\Sigma_k\cap \partial M_{\lambda}$ is a null singular point, by Theorem \ref{ThmGaussBonnetFrontsWithB} (see \eqref{GBminusformula}) we get \eqref{LambdaGeodesic}.

By the genericity of $R_m$ the front $(f_k, \nu_k)\big|_{M_{\frac{1}{2}}}$ satisfies the assumptions of Theorem \ref{ThmGaussBonnetFrontsWithB}. Since $\displaystyle\int_{\Sigma_k}\kappa_sd\tau = 2\int_{\Sigma_k\cap M_{\frac{1}{2}}}\kappa_sd\tau$, we get \eqref{CutInHalfPlus} and \eqref{CutInHalfMinus}.

The angle between initial vectors (see Definition \ref{DefInitial}) of the singular curve at $p$ and of the boundary curve at $p$ is $\alpha_+(p)$ (see Theorem \ref{ThmGaussBonnetFrontsWithB}). By Proposition \ref{PropPeakLimit} and Proposition \ref{PropLimit} we get \eqref{AngleFrontFormula}.
\end{proof}

Furthermore directly by \eqref{GBminusformula} we get the following proposition.

\begin{prop}
Let $k$ be an odd number. Let $R_m$ be a generic $m$ - rosette. Let $\mathcal{C}^+$ (respectively $\mathcal{C}^-$) be a simple regular curve in $M^+$ (respectively $M^-$) which is smoothly homotopic to $\{1\}\times S^1$ (respectively $\{0\}\times S^1$). If the orientations of $\mathcal{C}^+$, $\mathcal{C}^-$ are opposite then
\begin{align*}
\int_{\mathcal{C}^+}\kappa_{k, g}d\tau+\int_{\mathcal{C}^-}\kappa_{k, g}d\tau=0,
\end{align*}
where $d\tau$ denote the arc length measure.
\end{prop}

By Theorem \ref{ThmTotalGeodSing} we can get the relation between integrals of the curvature of the Centre Symmetry Set, the curvature of the rosette and the width of the rosette.

\begin{cor}
Let $k$ be an odd number and let $R_m$ be a generic $m$-rosette. Then
\begin{align}
\int_{R_m}&\kappa(\theta(s))\cdot\frac{w_k(\theta(s))}{\sqrt{1+w^2_k(\theta(s))}}ds=\\
\nonumber \int_{\Css_k(R_m)}&\kappa_{\Css_k(R_m)}(\theta(\ell))\cdot\frac{\big(\rho(\theta(\ell))+\rho(\theta(\ell)+k\pi)\big)\sqrt{1+w^2_k(\theta(\ell))}}{\big(1+w^2_k(\theta(\ell))+w'^2_k(\theta(\ell))\big)^{\frac{3}{2}}}d\ell,
\end{align}
where $s$ (respectively $\ell$) is the arc length parameter on $R_m$ (respectively on $\Css_m(R_m)$).
\end{cor}

\begin{thm}\label{ThmIntegralWidth}
Let $k$ be an odd number and let $R_m$ be a generic $m$-rosette. Then
\begin{align}\label{IntegralsWidth}
\int_0^{2m\pi}\frac{w_k(\theta)}{\sqrt{1+w_k^2(\theta)}}d\theta=\int_0^{2m\pi}\left(w_k(\theta)+w''_k(\theta)\right)\cdot\frac{\sqrt{1+w^2_k(\theta)}}{1+w^2_k(\theta)+w'^2_k(\theta)}d\theta.
\end{align}
\end{thm}
\begin{proof}
The proof is a straightforward use of (\ref{GeoCurvForm}), (\ref{CssKsRelation}) and the fact that \linebreak $\rho(\theta)+\rho(\theta+k\pi)=w_k(\theta)+w''_k(\theta)$.
\end{proof}

\begin{rem}
Since $w_k(\theta)=\sinh(C_1\theta+C_2)$ for $C_1,C_2\in\mathbb{R}$ is the general solution of
\begin{align}\label{DiffEquation}
\frac{w_k(\theta)}{\sqrt{1+w_k^2(\theta)}} =\left(w_k(\theta)+w''_k(\theta)\right)\cdot\frac{\sqrt{1+w^2_k(\theta)}}{1+w^2_k(\theta)+w'^2_k(\theta)},
\end{align}
the only periodic solution of (\ref{DiffEquation}) is a constant function. Therefore the relation (\ref{IntegralsWidth}) is naively fulfilled only for rosettes of constant $k$ - width.
\end{rem}

\begin{rem}
The condition that $w$ is of class $C^2(\mathbb{R})$ cannot be omitted.
We can consider the function $w(\theta)=1+\left|x-\pi\right|^3$  and the interval $[0,2\pi]$. One can check that relation (\ref{IntegralsWidth}) does not hold.
\end{rem}

\begin{rem}
By \eqref{EqWidth} the odd coefficients of the Fourier series of a width of an oval vanish. Thus a function $w(\theta)=2+\sin 3\theta$ is not a width of any oval but it satisfies the relation (\ref{IntegralsWidth}).
\end{rem}

\begin{con}\label{ConIntegralWidth}
Let $w:\mathbb{R}\to\mathbb{R}$ be a $2\pi$-periodic $C^2(\mathbb{R})$ function. Then $w$ satisfies the relation
\begin{align}\label{IntegralsWidthConj}
\int_0^{2\pi}\frac{w(\theta)}{\sqrt{1+w^2(\theta)}}d\theta=\int_0^{2\pi}\left(w(\theta)+w''(\theta)\right)\cdot\frac{\sqrt{1+w^2(\theta)}}{1+w^2(\theta)+w'^2(\theta)}d\theta.
\end{align}
\end{con}

In \cite{MS1, MSUY1} others invariants of cuspidal edges of fronts  are introduced. Let $(f,\nu):M\mapsto \mathbb R^3\times S^2$ be a front.  Let $\gamma$ be a singular curve near an $A_2$-point (a cuspidal edge)
and $\eta$ be a null direction along $\gamma$ such that the singular direction $\gamma'$ and the null direction $\eta$ form a positively oriented frame. We put $\hat{\gamma}=f\circ\gamma$, $f_{\eta}=df(\eta)$, $f_{\eta, \eta}=d(f_{\eta})(\eta)$, $f_{\eta, \eta, \eta}=d(f_{\eta, \eta})(\eta)$. Then  {\it the limiting normal curvature along} $\gamma$ is defined in the following way
\begin{align}
\label{CuspNormalCurvFormula}
\kappa_{\nu}(t) &=\frac{\left<\hat{\gamma}''(t), \nu\left(\gamma(t)\right)\right>}{|\hat{\gamma}'(t)|^2}.
\end{align}
The {\it cuspidal curvature along} $\gamma$ is defined as follows:
\begin{align}
\label{CuspCuspidalCurvFormula}
\kappa_{c}(t) &=\frac{|\hat{\gamma}(t)|^{\frac{3}{2}}\det\left(\hat{\gamma}(t), f_{\eta\eta}(\gamma(t)), f_{\eta\eta\eta}(\gamma(t))\right)}{\left|\hat{\gamma}(t)\times f_{\eta\eta}(\gamma(t))\right|^{\frac{5}{2}}}.
\end{align}
The {\it cusp-directional torsion} is defined by the formula
\begin{align}
\label{CuspDirTorsionFormula}
\kappa_t(t) &=\frac{\det\left(\hat{\gamma}', f_{\eta\eta}(\gamma), (f_{\eta\eta}(\gamma))'\right)}{\left|\hat{\gamma}'\times f_{\eta\eta}(\gamma)\right|^2}(t)-\frac{\det\left(\hat{\gamma}', f_{\eta\eta}(\gamma), \hat{\gamma}''\right)\cdot\left<\hat{\gamma}', f_{\eta\eta}(\gamma)\right>}{|\hat{\gamma}'|^2|\hat{\gamma}'\times f_{\eta\eta}(\gamma)|^2}(t).
\end{align}

In \cite{SUY} it was shown that a point $p$ is a generic cuspidal edge if and only if $\kappa_{\nu}(p)$ does not vanish.
The  curvature $\kappa_c$ is exactly the cuspidal curvature of the cusp of the plane
curve obtained as the intersection of the surface by the plane $H$, where $H$ is orthogonal to the tangential direction at a given cuspidal edge (\cite{MSUY1}).
For the geometrical meaning of the cusp-directional torsion (\ref{CuspDirTorsionFormula}) see Proposition 5.2 in \cite{MS1} and for global properties see Appendix A in \cite{MS1}. By straightforward calculations we obtain the following lemma.

\begin{lem}
Let $R_m$ be a generic $m$-rosette. Let $k$ be an odd number. Then the normal curvature $\kappa_{k, \nu}$, the cuspidal curvature $\kappa_{k, c}$ and the cusp-directional torsion $\kappa_{k, t}$ of the cuspidal edge of $\E_k(R_m)$ at a point $\displaystyle\left(\frac{\kappa(\theta)}{\kappa(\theta)+\kappa(\theta+k\pi)}, \theta\right)$ are given by the following formulas
\begin{align}
\label{CuspNormalCurv} \kappa_{k, \nu}(\theta) &\equiv 0, \\
\label{CuspCuspidalCurv} \kappa_{k, c}(\theta) &=\frac{2\sqrt{\kappa(\theta)\kappa(\theta+k\pi)\big(\kappa(\theta)+\kappa(\theta+k\pi)\big)}}{\sqrt{\left|\left(\frac{\kappa(\theta+k\pi)}{\kappa(\theta)}\right)'\right|}}\cdot\frac{\big(1+w_k^2(\theta)+w_k'^2(\theta)\big)^{\frac{3}{4}}}{\big(1+w_k^2(\theta)\big)^{\frac{5}{4}}},\\
\label{CuspDirTorsion} \kappa_{k, t}(\theta) &=-\frac{\big(\kappa(\theta)+\kappa(\theta+k\pi)\big)^2}{\kappa^2(\theta)\cdot\left(\frac{\kappa(\theta+k\pi)}{\kappa(\theta)}\right)'}\cdot\frac{1}{1+w_k^2(\theta)}.
\end{align}
\end{lem}

\begin{prop}\label{CorollariesCurv}
Let $R_m$ be a generic $m$-rosette. Let $k$ be an odd number. Then
\begin{enumerate}[(i)]
\item cuspidal edges of $\E_k(R_m)$ are not generic,
\item the mean curvature of $\E_k(R_m)$ is not bounded,
\item the total torsion of the image of singular curve $\hat\gamma_k(\theta)$ for $\theta\in[0,2k\pi]$ is equal to $2n\pi$ for some integer $n$, i.e.
\begin{align}\label{FormulaTotalTorsion}
\int_{\gamma_k}\tau_k(s)ds=2n\pi,
\end{align}
where $\gamma_k$ is the singular curve, $\tau_k$ is a torsion of $\hat{\gamma}_k$ and $s$ is the arc length parameter of $\hat{\gamma}_k$.
\end{enumerate}
\end{prop}
\begin{proof}
\begin{enumerate}[(i)]
\item It is a consequence of (\ref{CuspNormalCurv}).
\item Since $\kappa_{k, c}(p)\neq 0$ for any cuspidal edge $p\in\Sigma$, then by Proposition 2.8 in \cite{MSUY1} we get that the mean curvature of $\Css_k(R_m)$ is not bounded.
\item From Appendix A in \cite{MS1} we know that in our case there is the following equality
\begin{align*}
\int_{\gamma_k}\kappa_{k,t}(s)ds=\int_{\gamma_k}\tau_k(s)ds-2n\pi.
\end{align*}
It is easy to see that $\displaystyle \int_{\gamma_k}\kappa_{k,t}(s)ds=0$. Hence (\ref{FormulaTotalTorsion}) holds.
\end{enumerate}
\end{proof}

\begin{rem}
For the geometrical meaning of the number $n$ in Corollary \ref{CorollariesCurv}(iii)  see Appendix A in \cite{MS1}. In \cite{QL1} authors show that the total torsion of a closed line of curvature on a surface (i.e. a closed curve on a surface whose tangents are always in the direction of a principal curvature) is $l\pi$, where $l$ is an integer. Furthermore they show that if the total torsion of a closed curve is $l\pi$ for an integer $l$, then this curve can appear as a line of curvature on a surface and if $l$ is even, then it can appear as a line of curvature on a surface of genus $1$.
\end{rem}

\section*{Acknowledgments}

The authors thank Kentaro Saji and Zbigniew Szafraniec for fruitful discussions
and valuable comments.

\bibliographystyle{amsalpha}

\end{document}